\documentclass[a4paper,11pt]{article}
\usepackage{amsmath, amssymb, graphicx, parskip, amsthm, fancyhdr, dsfont, algorithm, algpseudocode, hyperref, bm, mleftright}
\algtext*{EndFor}

\usepackage{todonotes}

\newtheorem{thm}{Theorem}

\newtheorem{lem}{Lemma}
\newtheorem{cor}{Corollary}
\theoremstyle{definition}

\newtheorem{rmk}{Remark}

\newcommand{\bbP}{\mathbb{P}}
\newcommand{\N}{\mathbb{N}}
\newcommand{\E}{\mathbb{E}}
\newcommand{\R}{\mathbb{R}}
\newcommand{\X}{\mathbf{X}}
\newcommand{\w}{\mathbf{w}}
\renewcommand{\a}{\mathbf{a}}

\newcommand{\calF}{\mathcal{F}}
\newcommand{\Var}{\operatorname{Var}}

\newcommand{\osum}{\mathop{\sum\nolimits^{\phantom{}_{\star}}}\limits}

\title{Asymptotic genealogies of interacting particle systems with an application to sequential Monte Carlo}
\author{Jere Koskela \\
	\texttt{j.koskela@warwick.ac.uk} \\
	\small Department of Statistics \\
	\small University of Warwick \\
	\small Coventry CV4 7AL, UK
	\and
	Paul A.~Jenkins \\
	\texttt{p.jenkins@warwick.ac.uk} \\
	\small Department of Statistics \& \\ \small Department of Computer Science \\
	\small University of Warwick \\
	\small Coventry CV4 7AL, UK
	\and
	Adam M.~Johansen \\
	\texttt{a.m.johansen@warwick.ac.uk} \\
	\small Department of Statistics \\
	\small University of Warwick \\
	\small Coventry CV4 7AL, UK, \\
	\small \& the Alan Turing Institute \\
	\small 96 Euston Road \\
	\small London NW1 2DB, UK
	\and
	Dario Span\`{o} \\
	\texttt{d.spano@warwick.ac.uk} \\
	\small Department of Statistics \\
	\small University of Warwick \\
	\small Coventry CV4 7AL, UK
}
\date{\today}

\begin{document}

\maketitle

\begin{abstract}
We study weighted particle systems in which new generations are resampled from current particles with probabilities proportional to their weights. 
This covers a broad class of sequential Monte Carlo (SMC) methods, widely-used in applied statistics and cognate disciplines. 
We consider the genealogical tree embedded into such particle systems, and identify conditions, as well as an appropriate time-scaling, under which they converge to the Kingman $n$-coalescent in the infinite system size limit in the sense of finite-dimensional distributions. 
Thus, the tractable $n$-coalescent can be used to predict the shape and size of SMC genealogies, as we illustrate by characterising the limiting mean and variance of the tree height. 
SMC genealogies are known to be connected to algorithm performance, so that our results are likely to have applications in the design of new methods as well.
Our conditions for convergence are strong, but we show by simulation that they do not appear to be necessary.
\end{abstract}

\section{Introduction}\label{introduction}

Interacting particle systems (IPSs) are a broad class of stochastic models for phenomena in disciplines including physics, engineering, biology and finance.
Prominent examples are particle filters, particle methods and sequential Monte Carlo (SMC), which feature prominently in numerical approximation schemes for nonlinear filtering, as well as mean field approximation of Feynman--Kac flows.
For additional background we direct readers to \cite{DelMoral04}, and \cite{Doucet11}.

Central to these methods are discrete-time, evolving weighted particle systems.
Correlations between particles arise out of \emph{resampling}: a stochastic selection mechanism in which particles with high weight are typically replicated while those with low weight vanish, giving rise to an embedded genealogy.
The contribution of this paper is to identify conditions under which these genealogies converge to the Kingman $n$-coalescent \cite{Kingman82a} in the infinite system size limit under an appropriate rescaling of time, as well as to describe ways in which information about the limiting genealogy can be used to characterize particle systems in applications.

Consider a sequence of measurable spaces $( E_t, \mathcal{E}_t )_{ t \in \N }$, each associated to a Markov transition kernel $K_{ t + 1} : E_t \times \mathcal{ E }_{ t + 1 } \mapsto ( 0, \infty )$, and a nonnegative potential $g_{ t + 1 } : E_t \times E_{ t + 1 } \mapsto ( 0, \infty )$.
These correspond to state spaces, transition kernels, and the importance weight function of our IPS, respectively.

Let $\zeta_t^{ ( N ) } := \{ ( w_t^{ ( i ) }, X_t^{ ( i ) } ) \}_{ i = 1 }^N$ be a weighted $N$-particle system at time $t \in \N$, where each $X_t^{ ( i ) } \in E_t$, and the weights $w_t^{ ( i ) }$ are nonnegative and satisfy $\sum_{ i = 1 }^N w_t^{ ( i ) } = 1$.
Let $S$ be a resampling operator which acts on $\zeta_t^{ ( N ) }$ by assigning to each particle a random number of offspring.
The total number of offspring is fixed at $N$, and the mean number of offspring of particle $i \in \{ 1, \ldots, N \}$ is $N w_t^{ ( i ) }$.
All offspring are assigned an equal weight $1 / N$.
More concretely
\begin{equation*}
S \zeta_t^{ ( N ) } = \{ ( N^{ -1 }, X_t^{ ( a_t^{ ( i ) } ) } ) \}_{ i = 1 }^N,
\end{equation*}
where $a_t^{ ( i ) } = j$ if $j$ in $\zeta_t^{ ( N ) }$ is the parent of $i$ in $S \zeta_t^{ ( N ) }$.
Particles with low weight are randomly removed by having no offspring, while particles with high weight tend to have many offspring.
Algorithm \ref{multinomial_resampling} gives an example.

The step from time $t$ to time $t + 1$ is completed by propagating each particle in $S \zeta_t^{ ( N ) }$ independently through the transition kernel $K_{ t + 1 }$ to obtain particle locations $X_{ t + 1 }^{ ( i ) } \sim K_{ t + 1 }( X_t^{ ( a_t^{ ( i ) } ) }, \cdot )$.
Finally, each particle $i \in \{ 1, \ldots, N \}$ is assigned a weight proportional to the potential $g_{ t + 1 }$ evaluated at the locations of the particle and its parent, so that the full update is
\begin{equation*}
\zeta_t^{ ( N ) } \stackrel{ S }{ \mapsto } \{ N^{ -1 }, X_t^{ ( a_t^{ ( i ) } ) } \}_{ i = 1 }^N \stackrel{ K_{ t + 1 } }{ \mapsto } \Bigg\{ \frac{ g_{ t + 1 }( X_t^{ ( a_t^{ ( i ) } ) }, X_{ t + 1 }^{ ( i ) } ) }{ \sum_{ j = 1 }^N g_{ t + 1 }( X_t^{ ( a_t^{ ( j ) } ) }, X_{ t + 1 }^{ ( j ) } ) }, X_{ t + 1 }^{ ( i ) } \Bigg\}_{ i = 1 }^N,
\end{equation*}
where $g_0( X_{ -1 }, X_0 ) \equiv g_0( X_0 )$.
A specification is given in Algorithm \ref{particle_filter}. 
\begin{algorithm}[!ht]
\caption{Simulation of an IPS}
\label{particle_filter}
\begin{algorithmic}[1]
\Require Particle number $N$, Markov kernels $K_t$, potentials $g_t$, initial proposal distribution $\mu$.
\For{$i \in \{ 1, \ldots, N \}$} 
	\State Sample $X_0^{ ( i ) } \sim \mu$.
\EndFor
\For{$i \in \{ 1, \ldots, N \}$}
	\State Set $w_0^{ ( i ) } \gets \frac{ g_0( X_0^{ ( i ) } ) }{ g_0( X_0^{ ( 1 ) } ) + \cdots + g_0( X_0^{ ( N ) } ) }$.
\EndFor
\For{$t \in \{ 0, \ldots, T - 1 \}$}
	\State Sample $( a_t^{ ( 1 ) }, \ldots, a_t^{ ( N ) } ) \sim \operatorname{Resample}( w_t^{ ( 1 ) }, \ldots, w_t^{ ( N ) } )$.
	\For{$i \in \{ 1, \ldots, N \}$}
		\State Sample $X_{ t + 1 }^{ ( i ) } \sim K_{ t + 1 }( X_t^{ ( a_t^{ ( i ) } ) }, \cdot )$.
	\EndFor
	\For{$i \in \{ 1, \ldots, N \}$}
		\State Set $w_{ t + 1 }^{ ( i ) } \gets \frac{ g_{ t + 1 }\big( X_t^{ ( a_t^{ ( i ) } ) }, X_{ t + 1 }^{ ( i ) } \big) }{ g_{ t + 1 }\big( X_t^{ ( a_t^{ ( 1 ) } ) }, X_{ t + 1 }^{ ( 1 ) } \big) + \cdots + g_{ t + 1 }\big( X_t^{ ( a_t^{ ( N ) } ) }, X_{ t + 1 }^{ ( N ) } \big) }$.
	\EndFor
\EndFor
\end{algorithmic}
\end{algorithm}

There are many options for the resampling step \texttt{Resample} in line 6 of Algorithm \ref{particle_filter}.
The simplest is \emph{multinomial resampling}, given in Algorithm \ref{multinomial_resampling}.
It is analytically tractable, but suboptimal in terms of Monte Carlo variance.
Popular alternatives include \emph{residual}, \emph{stratified}, and \emph{systematic} resampling, which yield algorithms with lower variance \cite{douc05}.
We prove our main theorem under multinomial resampling, and show by simulation that its conclusions also hold for these three alternatives, at least in our simple example. 
Other schemes which have been studied theoretically, but seem less widely used in applications, include the tree-based branching approximation of \cite{smc:theory:CL97} and the McKean interpretation approach by \cite[Section 2.5.3]{DelMoral04}.
\begin{algorithm}[!ht]
\caption{Multinomial resampling}
\label{multinomial_resampling}
\begin{algorithmic}[1]
\Require Normalised particle weights $\{ w_t^{ ( i ) } \}_{ i = 1 }^N$.
\For{$i \in \{ 1, \ldots, N \}$}
	\State Sample $a_t^{ ( i ) } \sim \operatorname{Categorical}( w_t^{ ( 1 ) }, \ldots, w_t^{ ( N ) } )$.
\EndFor
\State \Return $( a_t^{ ( 1 ) }, \ldots, a_t^{ ( N ) } )$.
\end{algorithmic}
\end{algorithm}

An example system of this form is particle filters: a class of IPSs for approximating integrals.
For an overview, see, for example, \cite{Doucet11}, and references therein.
They are suited to settings in which expectations are computed with respect to the law of a latent, discrete-time Markov process conditioned on a sequence of observations, for example, state space models.
A state space model is a Markov chain $( X_t, Y_t )_{ t \geq 0 }$ on $\bm{ E } \times \bm{ F } \subseteq \R^d \times \R^p$ with transition density $p( x, x' )$ and emission density $\psi( x, y )$, both with respect to dominating measures abusively denoted $dx'$ and $dy$, respectively.
The spaces $( \bm{ E }, \bm{ F } )$ could be replaced with more general Polish spaces, but we focus on (subsets of) $\R^d \times \R^p$ for concreteness.
Typically, $X_t$ is an unobserved state at time $t$, and $Y_t$ is a noisy or incomplete observation of the state.

Let $\bm{ y } := ( y_0, \ldots, y_T )$ be a vector of noisy, conditionally independent observations from the emission density $\psi( x, y )$ given an unobserved state trajectory $\bm{ x } := ( x_0, \ldots, x_T )$.
Functionals of the \emph{smoothing distribution}, 
\begin{align*}
P( \bm{ x } | \bm{ y } ) d\bm{ x } &:= \bbP( \bm{ X } \in d\bm{ x } | \bm{ Y } = \bm{ y } ) 
\propto \prod_{ t = 0 }^T p( x_{ t - 1 }, x_t ) \psi( x_t, y_t ) dx_{ 0 : T },
\end{align*}
where we abuse notation and denote the law of $X_0$ by $p( x_{ -1 }, x_0 ) dx_0$, are typically intractable, and often approximated by particle filters.
The simplest is the algorithm introduced by \cite{GSS93}, usually known as the bootstrap particle filter, although the term ``(interacting) particle filter" seems first to have been used by \cite{DelMoral96}.
It fits into our IPS framework by taking
$K_{ t + 1 }( x_t, dx_{ t + 1 } ) = p( x_t, x_{ t + 1 } ) d x_{ t + 1 }$ and 
$g_{ t + 1 }( x_t, x_{ t + 1 } ) = \psi( x_{ t + 1 }, y_{ t + 1 } )$
for a fixed observation sequence in Algorithm \ref{particle_filter}.
Other particle filters can be considered by introducing a proposal density $q_{ t + 1 }( x_t, x_{ t + 1 } ) dx_{ t + 1 }$ with $\operatorname{supp}( p( x, \cdot ) ) \subseteq \operatorname{supp}( q_t( x, \cdot ) )$ for every $x$ and $t$, and setting 
\begin{align*}
K_{ t + 1 }( x_t, dx_{ t + 1 } ) &= q_{ t + 1 }( x_t, x_{ t + 1 } ) d x_{ t + 1 }, \\
g_{ t + 1 }( x_t, x_{ t + 1 } ) &= \frac{ p( x_t, x_{ t + 1 } ) \psi( x_{ t + 1 }, y_{ t + 1 } ) }{ q_{ t + 1 }( x_t, x_{ t + 1 } ) }. 
\end{align*}
Our IPS framework also includes more general SMC algorithms, such as SMC samplers \cite{DelMoral06} among others, see, for example, \cite[Chapter 12]{DelMoral04}.

There is a genealogy embedded in Algorithm \ref{particle_filter}.
Consider $\zeta_t^{ ( N ) }$ at a fixed time $t$.
Tracing the ancestor indices $( a_t^{ ( 1 ) }, \ldots a_t^{ ( N ) } )$ backwards in time results in a coalescing forest of lineages.
The forest forms a tree once the most recent common ancestor (MRCA) of all particles is reached, provided that happens before reaching the initial time 0.
Our main result (Theorem \ref{general_thm}) shows that, under certain conditions and an appropriate time-rescaling, functionals of this genealogy depending upon finite numbers of leaves converge to corresponding functionals of the Kingman $n$-coalescent \cite{Kingman82a} as $N \rightarrow \infty$, in the sense of finite dimensional distributions.
Hence, the tractable Kingman $n$-coalescent can be used to approximate functionals of SMC genealogies off-line, \emph{before} the algorithm has been run.
In particular, we show that the expected number of time steps from the leaves to the MRCA scales linearly in $N$ for any finite number of leaves.
This compares to an $N \log N$ upper bound of \cite{Jacob15} for the genealogy of all $N$ leaves.
We also provide scaling expressions for the variance of the number of time steps to the MRCA (see Corollary \ref{time_scale_cor}).
Our result applies to the marginal genealogical tree structure, marginalised over locations and weights of particles.

The conditions for convergence to the Kingman $n$-coalescent can be satisfied by SMC algorithms under strong but standard conditions \eqref{weight_bound} and \eqref{mixing_bound}, used to control the fluctuations of family sizes.
We expect that they can be relaxed, and simulations in Section \ref{numerics} also suggest they are not necessary.

Mergers of lineages into common ancestors result in \emph{path degeneracy} in SMC algorithms.
Quantities of interest depending on times $t < T$ will be estimated using fewer than $N$ particles, resulting in high variance.
Path degeneracy can be reduced by resampling less frequently, but this increases variance across time.
The extreme case of no resampling results in no path degeneracy, but yields estimators whose variances increase exponentially in $T$ in most cases: a phenomenon known as \emph{weight degeneracy} \cite{Kong94}.
Intermediate strategies balance these two difficulties \cite[Section 4.2]{Liu95}.
Our limiting genealogies are the only method of which we are aware that can provide \emph{a priori} estimates of path degeneracy, and thus represent an important tool in minimising both degeneracies simultaneously. 
This is particularly important in the conditional SMC update in the particle Gibbs algorithm \cite{Andrieu10}, where retaining multiple trajectories across a fixed time window is essential.

Genealogies of SMC algorithms have found numerous applications, for example \cite{Cerou11,smc:theory:CL13,DelMoral01b, DelMoral05, DelMoral07, DelMoral09,Lee15, smc:theory:OD17}; see also \cite[Section 3.4]{DelMoral04} and \cite[Chapter 15]{DelMoral13}.
Our explicit description of the limiting genealogical process is new for SMC, and has the potential to build on any of these works, as well as generate new ones in the IPS and SMC fields.
We also demonstrate that non-Markovian prelimiting genealogies lie in the domain of attraction of Kingman's $n$-coalescent, in contrast with earlier results which assume Markovianity \cite[Theorem 1]{Moehle98}.
This extension is significant since reverse-time genealogies of Markov IPSs are not Markov processes in general.

The rest of this paper is structured as follows.
In Section \ref{theorem} we state and prove the convergence of genealogies of IPSs to the Kingman $n$-coalescent.
Section \ref{numerics} presents a simulation showing that the scalings predicted by our convergence result hold for an example outside our technical assumptions, and Section \ref{discussion} concludes with a discussion.
An Appendix contains some technical calculations.
We conclude this section by summarising notation.

Let $( x )_b := x ( x - 1 ) \ldots ( x - b + 1 )$ be the falling factorial. 
We adopt the conventions $\sum_{ \emptyset } = 0$, $\prod_{ \emptyset } = 1$.
When a sum is written over a vector of indices, and that vector is of length 0, it should be interpreted as the identity operator; where this convention might hold, we emphasize it by writing $\osum$. 
The statement $f( N ) = O( g( N ) )$ (resp., $o( g( N ) )$) means $\limsup_{ N \rightarrow \infty } | f( N ) / g( N ) | < \infty$, (resp., $ = 0$), and thus corresponds to the usual Landau big-O (resp., little-o) notation.
For an integer $n \in \N$, we define $[ n ] := \{ 1, \ldots, n \}$ with $[0] := \emptyset$, and for a finite set $A$, we let $\Pi_n( A )$ denote the set of partitions of $A$ into at most $n$ nonempty blocks, with $\Pi_n( \emptyset ) := ( \emptyset, \ldots, \emptyset)$.
For a partition $\xi$, $| \xi |$ denotes the number of blocks in $\xi$, and $\bm{ x } := ( x_1, \ldots, x_N )$, where the length of the vector will be clear from context.
For a vector $\bm{ x }$, $| \bm{ x } |$ denotes the $L^1$-norm.

\section{The convergence theorem}\label{theorem}

It will be convenient to express our IPS in reverse time by denoting the initial time in Algorithm \ref{particle_filter} by $T$ and the terminal time by 0, and to describe the genealogy in terms of a partition-valued family of processes $( G_t^{ ( n, N ) } )_{ t = 0 }^T$ indexed by $n \leq N$, where $n$ denotes the number of observed leaves (time 0 particles) in a system with $N$ particles.
The process $( G_t^{ ( n, N ) } )_{ t = 0 }^T$ is defined in terms of the underlying IPS via its initial condition $G_0^{ ( n, N ) } = \{ \{ 1 \}, \ldots, \{ n \} \}$, and its dynamics, which are driven by the requirement that $i \neq j \in [ n ]$ belong to the same block in $G_t^{ ( n, N ) }$ if leaves $i$ and $j$ have a common ancestor at time $t$.
Boundary problems will be avoided by ensuring that $T \rightarrow \infty$ in our rescaled system as $N \rightarrow \infty$.
\vskip 11pt
\begin{rmk}
Our genealogical process $( G_t^{ ( n, N ) } )_{ t = 0 }^T$ evolves on a space which tracks the ancestral relationships of the observed particles but not their states. The process is a projection of the time reversal of the historical process of \cite{DelMoral01b}, in which particle locations have been marginalised over. A consequence of this is that $( G_t^{ ( n, N ) } )_{ t = 0 }^T$ is not Markovian in general. 
Indeed, Markovianity fails even for the forward-time evolution of lineages after removing locations.
Genealogical processes typically do track locations in the SMC literature, whereas our marginal formulation is standard in genetics \cite{Moehle98}.
Our $( G_t^{ ( n, N ) } )_{ t = 0 }^T$ also coincides with the random ancestral forest of \cite{DelMoral09}, who showed that the combinatorial structure of the reverse-time genealogy decouples from the particle locations in the case of \emph{neutral} models ($g_t( x, y ) \equiv 1$).
Our contribution is to prove the same decoupling for suitably rescaled nonneutral models in the $N \rightarrow \infty$ limit, and to identify the limiting process.
\end{rmk}

Genealogical processes are a powerful tool in population genetics, where the genealogical tree is viewed as missing data to be imputed, or an object of inference in its own right.
A common large population limit is given by the Kingman $n$-coalescent \cite{Kingman82a}, which is also a partition-valued stochastic process evolving in reverse time.
The initial condition is the singleton partition $\{ \{ 1 \}, \ldots, \{ n \} \}$.
Each pair of blocks then merges to a common ancestor independently at rate 1, forming a death process on the total number of blocks with rate $\binom{ k }{ 2 }$ when there are $k$ blocks.
More formally, the generator of the Kingman $n$-coalescent, $Q = ( q_{ \xi \eta } )_{ \xi \eta }$, is the square matrix with a row and column corresponding to each partition of $[ n ]$, and
\begin{equation*}
q_{ \xi \eta } = \begin{cases}
1 &\text{ if } \eta \text{ can be obtained from } \xi \text{ by merging two blocks,}\\
- \binom{ | \xi | }{ 2 } &\text{ if } \eta = \xi, 
\end{cases}
\end{equation*}
and 0 otherwise.
The dynamics terminate once the process hits the MRCA, that is, the trivial partition $\{ \{ 1, \ldots, n \}\}$.
See \cite{wakeley09} for an introduction to coalescent processes and their use in population genetics.

Let $\nu_t^{ ( i ) }$ denote the number of offspring that particle $i$ at time $t$ has at time $t - 1$.
The following standing assumption will be central to our results.
\vskip 11pt
\textbf{Standing assumption}: The conditional distribution of parental indices given offspring counts, $\a_t | \bm{ \nu }_t$, is uniform over all vectors which satisfy $\nu_t^{ ( i ) } = \#\{ j \in [ N ] : a_t^{ ( j ) } = i \}$ for each $i \in [ N ]$.
\vskip 11pt
\begin{rmk}
The standing assumption concerns the marginal distribution of parental assignments without particle locations. Conditionally uniform assignment given locations would be much stronger.
A sufficient condition for the standing assumption is exchangeability of the \texttt{Resample} mechanism in line 6 of Algorithm \ref{particle_filter}.
However, \cite[page 446]{Moehle98} provides an example of a nonexchangeable particle system which still satisfies the standing assumption.
For SMC, multinomial and residual resampling can be implemented in ways which satisfy the standing assumption, and any resampling scheme can be made to satisfy it by applying a uniformly sampled permutation to each realisation of ancestor indices.
This technique was suggested in \cite[page 290]{Andrieu10}.
\end{rmk}
Let $\xi$ and $\eta$ be partitions of $[ n ]$ with blocks ordered by the least element in each block, and with $\eta$ obtained from $\xi$ by merging some subsets of blocks.
For $i \in [ | \eta | ]$, let $b_i$ be the number of blocks of $\xi$ that have been merged to form block $i$ in $\eta$, so that $b_1 + \ldots + b_{ | \eta | } = | \xi |$, and define
\begin{equation}\label{kingman_formula}
p_{ \xi \eta }( t ) := \frac{ 1 }{ ( N )_{ | \xi | } } \sum_{ \substack{ i_1 \neq \ldots \neq i_{ | \eta | } = 1 \\ \text{all distinct} } }^N ( \nu_t^{ ( i_1 ) } )_{ b_1 } \ldots ( \nu_t^{ ( i_{ | \eta | } ) } )_{ b_{ | \eta | } }
\end{equation}
as the conditional transition probability of the genealogical process from state $\xi$ at time $t - 1$ to state $\eta$ at time $t$, given $\bm{ \nu }_t$ (suppressed from the notation).
The interpretation of \eqref{kingman_formula} as a conditional transition probability of $G_t^{ ( n, N ) }$ is justified by the standing assumption by associating offspring with balls, parents with boxes, and merger events with two or more balls occupying the same box.
Finally, let $P_N( t ) = ( p_{ \xi \eta }( t ) )_{ \xi, \eta }$ denote the corresponding conditional transition probability matrix. 

Now let the conditional probability (resp., upper bound on conditional probability) given $\bm{ \nu }_t$ of two (resp., more than two) lineages at time $t - 1$ coalescing at time $t$ be denoted by
\begin{align*}
c_N( t ) &:= \frac{ 1 }{ ( N )_2 }\sum_{ i = 1 }^N ( \nu_t^{ ( i ) } )_2, \\
D_N( t ) &:= \frac{ 1 }{ N ( N )_2 } \sum_{ i = 1 }^N ( \nu_t^{ ( i ) } )_2 \Bigg( \nu_t^{ ( i ) } + \frac{ 1 }{ N } \sum_{ j \neq i }^N ( \nu_t^{ ( j ) } )^2 \Bigg).
\end{align*}
The interpretations as (upper bounds on) conditional probabilities are justified by the standing assumption, and Lemma \ref{kingman_bounds} below.
For $t > 0$, let $\tau_N( t )$ be the time change driven by the random offspring counts
 \begin{equation*}
\tau_N( t ) := \min\Bigg\{ s \geq 1 : \sum_{ r = 1 }^s c_N( r ) \geq t \Bigg\}.
\end{equation*}
\vskip 11pt
\begin{rmk}\label{ess_remark}
The quantity $c_N( t )$ will play the role of a merger rate between pairs of particles at time $t$.
For multinomial resampling,
\begin{equation*}
\E[ c_N( t ) | \w_t ] = \sum_{ i = 1 }^N ( w_t^{ ( i ) } )^2 = \frac{ 1 }{ \operatorname{ESS}( t ) } \approx \frac{ 1 }{ N }\Bigg( 1 + \Var\Big( \frac{ g_t( X_{ t + 1 }, X_t ) }{ \E[ g_t( X_{ t + 1 }, X_t ) ] } \Big) \Bigg),
\end{equation*}
where $\operatorname{ESS}( t )$ is the effective sample size of \cite{Kong94}, who also justify the approximation of the random left-hand side by the deterministic right-hand side.
The indices of $X_{ t + 1 }$ and $X_t$ are flipped due to the time reversal described at the top of this section.
Thus the coalescence rate is high when the unnormalised importance weights have high relative variance and vice versa. 
\end{rmk}
\vskip 11pt
\begin{lem}\label{kingman_bounds}
Suppose that the standing assumption holds.

\textbf{Case 1}: For any partition $\xi$ and a sufficiently large $N$, we have the inequalities
\begin{align*}
p_{ \xi \xi }( t ) &\geq 1 - \binom{ | \xi | }{ 2 } \{ 1 + O( N^{ -1 } ) \} \Bigg[ \frac{ ( 3 | \xi | - 1 )( | \xi | - 2 ) }{ 6 N^2 } + c_N( t ) \Bigg], \\
p_{ \xi \xi }( t ) &\leq 1 - \binom{ | \xi | }{ 2 } \{ 1 + O( N^{ -1 } ) \} \Bigg[ c_N( t ) - \binom{ | \xi | - 1 }{ 2 } D_N( t ) \Bigg].
\end{align*}

\textbf{Case 2}: Let $\eta$ be obtained from $\xi$ by merging exactly two blocks.
Then
\begin{equation*}
c_N( t ) - \binom{ | \xi | - 2 }{ 2 } \{ 1 + O( N^{ -1 } ) \} D_N( t ) \leq p_{ \xi \eta }( t ) \leq \{ 1 + O( N^{ -1 } ) \} c_N( t ).
\end{equation*}

\textbf{Case 3}: For any $\eta$ obtained from $\xi$ by one or more mergers involving more than two blocks in total, we have
\begin{equation*}
p_{ \xi \eta }( t ) \leq \binom{ | \xi | - 2 }{ 2 } \{ 1 + O( N^{ -1 } ) \} D_N( t ).
\end{equation*}
\end{lem}
\begin{proof}
We begin by proving Case 1 by setting $\xi = \eta$ in \eqref{kingman_formula}.
A multinomial expansion in reverse yields
\begin{align*}
p_{ \xi \xi }( t ) &\geq \frac{ 1 }{ ( N )_{ | \xi | } } \Bigg\{ \Bigg( \sum_{ i = 1 }^N \nu_t^{ ( i ) } \Bigg)^{ | \xi | } - \binom{ | \xi | }{ 2 } \sum_{ i = 1 }^N ( \nu_t^{ ( i ) } )^2 \Bigg( \sum_{ j = 1 }^N \nu_t^{ ( j ) } \Bigg)^{ | \xi | - 2 } \Bigg\} \\
&= \frac{ N^{ | \xi | } }{ ( N )_{ | \xi | } } \Bigg[ 1 - \binom{ | \xi | }{ 2 } \frac{ 1 }{ N^2 } \sum_{ i = 1 }^N ( \nu_t^{ ( i ) } )^2 \Bigg] \\
&= 1 - \binom{ | \xi | }{ 2 } \{ 1 + O( N^{ -1 } ) \} \Bigg[ \frac{ ( 3 | \xi | - 1 ) ( | \xi | - 2 ) }{ 6 N^2 } + c_N( t ) \Bigg],
\end{align*}
where the second and third equalities follow from $\sum_{ i = 1 }^N \nu_t^{ ( i ) } \equiv N$, and 
\begin{align}
\frac{ N^{ | \xi | } }{ ( N )_{ | \xi | } } &= \Big( 1 - \frac{ 1 }{ N } \Big)^{ -1 } \cdots \Big( 1 - \frac{ | \xi | - 1 }{ N } \Big)^{ -1 } \nonumber \\
&= \Bigg( 1 - \frac{ 1 }{ N } \sum_{ k = 1 }^{ | \xi | - 1 } k + \frac{ 1 }{ N^2 } \sum_{ k = 1 }^{ | \xi | - 1 } \sum_{ m \neq k } k m + O( N^{ -3 } ) \Bigg)^{ -1 } \nonumber \\
&=\Bigg( 1 - \binom{ | \xi | }{ 2 }\frac{ 1 }{ N } +  \binom{ | \xi | }{ 2 } \Bigg[ \binom{ | \xi | }{ 2 } - \frac{ 2 | \xi | - 1 }{ 3 } \Bigg] \frac{ 1 }{ N^2 } + O( N^{ -3 } )\Bigg)^{ -1 } \nonumber \\
&= 1 + \binom{ | \xi | }{ 2 } \frac{ 1 }{ N } + \binom{ | \xi | }{ 2 } \frac{ 2 | \xi | - 1 }{ 3 N^2 } \{ 1 + O( N^{ -1 } ) \}, \label{falling_factorial_expansion}
\end{align}
where the final step follows from a Taylor expansion of $f(x) = (1 + x)^{-1 }$ around $x = 0$ to second order.
In the other direction,
\begin{equation*}
p_{ \xi \xi }( t )  \leq \frac{ 1 }{ ( N )_{ | \xi | } } \Bigg\{\Bigg( \sum_{ i = 1 }^N \nu_t^{ ( i ) } \Bigg)^{ | \xi | } - \binom{ | \xi | }{ 2 } \sum_{ i = 1 }^N ( \nu_t^{ ( i ) } )^2 \sum_{ \substack{ i_3 \neq \ldots \neq i_{ | \xi | } = 1 \\ \text{ all distinct \& } \neq i } }^N \nu_t^{ ( i_3 ) } \ldots \nu_t^{ ( i_{ | \xi | } ) } \Bigg\}.
\end{equation*}
Applying the previous lower bound to the $( | \xi | - 2 )$-fold sum \cite[last two displays on page 442]{Moehle98} and using 
\begin{align*}
( \nu_t^{ ( i ) } )^2 &= ( \nu_t^{ ( i ) } )_2 + \nu_t^{ ( i ) }, \\
( \nu_t^{ ( i ) } )^3 &= ( \nu_t^{ ( i ) } )_2 \nu_t^{ ( i ) } + ( \nu_t^{ ( i ) } )_2 + \nu_t^{ ( i ) }, \\
( \nu_t^{ ( i ) } )^2  ( \nu_t^{ ( j ) } )^2 &= ( \nu_t^{ ( i ) } )_2 ( \nu_t^{ ( j ) } )^2 + \nu_t^{ ( i ) } ( \nu_t^{ ( j ) } )_2 + \nu_t^{ ( i ) } \nu_t^{ ( j ) },
\end{align*}
gives
\begin{align*}
p_{ \xi \xi }( t ) &\leq \frac{ N^{ | \xi | } }{ ( N )_{ | \xi | } } \Bigg\{ 1 - \frac{ 1 }{ N } \binom{ | \xi | }{ 2 } + \frac{ 1 }{ N^2 } \binom{ | \xi | }{ 2 } \binom{ | \xi | - 2 }{ 2 } \\
&\phantom{\leq \frac{ N^{ | \xi | } }{ ( N )_{ | \xi | } } \Bigg\{}- \binom{ | \xi | }{ 2 } [ 1 + O( N^{ -1 } ) ] \Bigg[ c_N( t ) - \binom{ | \xi | - 1 }{ 2 } D_N( t ) \Bigg] \Bigg\},
\end{align*}
whereupon using \eqref{falling_factorial_expansion} and expanding the product shows that the claimed bound holds for large enough $N$.

The proof of Case 2 is essentially identical to calculations in \cite[pages 442--443]{Moehle98}, and is omitted.

For Case 3, we have
\begin{align*}
p_{ \xi \eta }( t ) &\leq \frac{ 1 }{ ( N )_{ | \xi | } } \sum_{ i = 1 }^N ( \nu_t^{ ( i ) } )_2 \Bigg\{ \Bigg( \sum_{ j = 1 }^N \nu_t^{ ( j ) } \Bigg)^{ | \xi | - 2 } - \sum_{ \substack{ j_1 \neq \ldots \neq j_{ | \xi | - 2 } = 1 \\ \text{all distinct \& } \neq i } }^N \nu_t^{ ( j_1 ) } \ldots \nu_t^{ ( j_{ | \xi | - 2 } ) } \Bigg\} \\
&\leq \frac{ 1 }{ ( N )_{ | \xi | } } \sum_{ i = 1 }^N ( \nu_t^{ ( i ) } )_2 \Bigg\{ N^{ | \xi | - 2 } - \Bigg( \sum_{ j \neq i }^N \nu_t^{ ( j ) } \Bigg)^{ | \xi | - 2 } + \binom{ | \xi | - 2 }{ 2 } \sum_{ j \neq i }^N ( \nu_t^{ ( j ) } )^2 \Bigg( \sum_{ k \neq i } \nu_t^{ ( k ) } \Bigg)^{ | \xi | - 4 } \Bigg\} \\
&\leq \frac{ 1 }{ ( N )_{ | \xi | } } \sum_{ i = 1 }^N ( \nu_t^{ ( i ) } )_2 \Bigg\{ N^{ | \xi | - 2 } - ( N - \nu_t^{ ( i ) } )^{ | \xi | - 2 } + \binom{ | \xi | - 2 }{ 2 } \sum_{ j \neq i }^N ( \nu_t^{ ( j ) } )^2 N^{ | \xi | - 4 } \Bigg\} \\
&\leq \frac{ 1 }{ ( N )_{ | \xi | } } \sum_{ i = 1 }^N ( \nu_t^{ ( i ) } )_2 \Bigg\{ ( | \xi | - 2 ) \nu_t^{ ( i ) } N^{ | \xi | - 3 } + \binom{ | \xi | - 2 }{ 2 } \sum_{ j \neq i }^N ( \nu_t^{ ( j ) } )^2 N^{ | \xi | - 4 } \Bigg\},
\end{align*}
where we have used the Bernoulli inequality in the last step.
The result follows from \eqref{falling_factorial_expansion}.
\end{proof}

\begin{thm}\label{general_thm}
Fix $n \leq N$ as the observed number of particles from the output of an IPS with $N$ particles, and suppose that the standing assumption holds.
Suppose also that for any $0 \leq s < t < \infty$, we have
\begin{gather}\label{big_merger_bound}
\lim_{ N \rightarrow \infty } \E\Bigg[ \sum_{ r = \tau_N( s ) + 1 }^{ \tau_N( t ) } D_N( r ) \Bigg] = 0, \\
\label{binary_bound}
\lim_{ N \rightarrow \infty } \E[ c_N( t ) ] = 0, \\
\label{binary_bound_2}
\lim_{ N \rightarrow \infty } \E\Bigg[ \sum_{ r = \tau_N( s ) + 1 }^{ \tau_N( t ) } c_N( r )^2 \Bigg] = 0, \\
\label{tau_bound}
\text{and}\qquad\qquad\E[ \tau_N( t ) - \tau_N( s ) ] \leq C_{ t, s } N,\qquad\qquad\phantom{\text{and}}
\end{gather}
for some constant $C_{ t, s } > 0$ that is independent of $N$.
Then $( G_{ \tau_N( t ) }^{ ( n, N ) } )_{ t \geq 0 }$ converges to the Kingman $n$-coalescent in the sense of finite-dimensional distributions as $N \rightarrow \infty$. 
\end{thm}
\vskip 11pt
\begin{rmk}
While unbiasedness of resampling, that is, $\E[ \nu_t^{ ( i ) } | w_t^{ ( i ) } ] = N w_t^{ ( i ) }$, is part of the definition of our IPS, it is not required for Theorem \ref{general_thm}.
The key ingredients are the symmetry in the standing assumption, and the control over moments of orders up to four implicit in \eqref{big_merger_bound} -- \eqref{tau_bound}.
Loosely, condition \eqref{binary_bound} implies that $( G_{ \tau_N( t ) }^{ ( n, N ) } )_{ t \geq 0 }$ converges to a continuous-time process, and together with \eqref{binary_bound_2} ensures that binary mergers happen at the required unit rate.
Condition \eqref{big_merger_bound} ensures that mergers involving more than two lineages happen on a slower timescale than binary mergers, and \eqref{tau_bound} controls the speed with which the convergence in \eqref{binary_bound} takes place.

Before proving Theorem \ref{general_thm}, we define the filtration
$\calF_t := \sigma( \bm{ \nu }_s; 1 \leq s \leq t )$
and state Lemma \ref{optional_stopping_lemma}, which we will need for the proof, as well as for our subsequent results.
The proof of Lemma \ref{optional_stopping_lemma} is given in the Appendix.
\vskip 11pt
\begin{lem}\label{optional_stopping_lemma}
For any $0 \leq s \leq t < \infty$,
\begin{equation*}
\E\Bigg[ \sum_{ r = \tau_N( s ) + 1 }^{ \tau_N( t ) } c_N( r ) \Bigg] = \E\Bigg[ \sum_{ r = \tau_N( s ) + 1 }^{ \tau_N( t ) } \E[ c_N( r ) | \calF_{ r - 1 } ] \Bigg].
\end{equation*}
\end{lem}
\vskip 11pt
\begin{rmk}
It will be clear from the proof that the choice of summand, $c_N( r )$, is not special in Lemma \ref{optional_stopping_lemma}.
Any family of bounded functions depending on family sizes at only a single time could be substituted, and we will also apply the same result to $c_N( r )^2$, $D_N( r )$, as well as indicators involving them.
\end{rmk}

\end{rmk}
\begin{proof}[Proof of Theorem \ref{general_thm}]
For $k \geq 1$ and $0 \leq t_1 < \ldots < t_k < \infty$, the finite-dimensional distributions of the process $( G_{ \tau_N( t ) }^{ ( n, N ) } )_{ t \geq 0 }$ have the form
\begin{align*}
&\bbP\Bigg( G_{ \tau_N( t_1 ) }^{ ( n, N ) } = \eta_1, \ldots, G_{ \tau_N( t_k ) }^{ ( n, N ) } = \eta_k | G_{ \tau_N( t_0 ) }^{ ( n, N ) } = \eta_0 \Bigg) \\
&= \E\Bigg[ \prod_{ d = 1 }^k \Bigg\{ \prod_{ t = \tau_N( t_{ d - 1 } ) + 1 }^{ \tau_N( t_{ d } ) } P_N( t ) \Bigg\}_{ \eta_{ d - 1 }, \eta_d } \Bigg] =: \E\Bigg[ \prod_{ d = 1 }^k \chi_d^* \Bigg],
\end{align*}
where $\eta_0, \eta_1, \ldots, \eta_k$ is a sequence of partitions in which the blocks of $\eta_d$ are obtained by merging some subsets of blocks of $\eta_{ d - 1 }$, or $\eta_d = \eta_{ d - 1 }$, and where $P_N( r )$ is the conditional transition matrix given family sizes defined below \eqref{kingman_formula}.
The probability associated with any other sequence of partitions is zero.
We will prove the result by bounding these finite-dimensional distributions both above and below by those of a Kingman $n$-coalescent.

Consider a transition between $\eta_{ d - 1 }$ and $\eta_d$ at respective times $\tau_N( t_{ d - 1 } )$ and $\tau_N( t_d )$.
The corresponding conditional transition probability given offspring counts $\bm{ \nu }_{ \tau_N( t_{ d - 1 } ) + 1 }, \ldots, \bm{ \nu }_{ \tau_N( t_d ) }$ can be written
\begin{align*}
\chi_d^* = \sum_{ \bm{ \xi } \in \eta_{ d - 1 } \leadsto \eta_d } \prod_{ t = \tau_N( t_{ d - 1 } ) + 1 }^{ \tau_N( t_d ) } p_{ \xi_{ t - 1 } \xi_{ t } }( t ),
\end{align*}
where the sum on the right-hand side is over all paths from $\eta_{ d - 1 }$ to $\eta_d$ of the requisite length,
$\bm{ \xi } = ( \eta_{ d - 1 }, \xi_{ \tau_N( t_{ d - 1 } ) + 1 }, \ldots, \xi_{ \tau_N( t_d ) - 1 }, \eta_d )$,
where each successive element of $\bm{ \xi }$ either equals its predecessor, or is obtained from its predecessor by merging some subsets of blocks.
By decomposing based on $\alpha \geq 0$ (the number of times between $\eta_{ d - 1 }$ and $\eta_d$ in which mergers occur), as well as based on whether each merger involves exactly two lineages or more than two lineages, we can use Lemma \ref{kingman_bounds}, Cases 2 and 3 to upper bound the conditional transition probability given family sizes by
\begin{align*}
\chi_d^* &\leq \sum_{ \alpha = 1 }^{ | \eta_{ d - 1 } | - | \eta_d | } ( 1 + O( N^{ -1 } ) )^{ \alpha } \sum_{ ( \lambda, \mu ) \in \Pi_2( [ \alpha ] ) } \\
&\osum_{ s_1 < \ldots < s_{ \alpha } = \tau_N( t_{ d - 1 } ) + 1 }^{ \tau_N( t_d ) }  \Bigg\{ \prod_{ r \in \lambda } \binom{ n }{ 2 } c_N( s_r ) \Bigg\} \Bigg\{ \prod_{ r \in \mu } n^n \binom{ n - 2 }{ 2 } D_N( s_r ) \Bigg\},
\end{align*}
and $\chi_d^* \leq 1$ if $\alpha = 0$, where $n^n$ is an upper bound on the number of arrangements of at most $n$ lineages into at most $n$ mergers.
A further expansion shows that the product of $k$ transition probabilities has the following bound when $\alpha_d \geq 1$ for each $d \in [k]$:
\begin{align}
\lim_{ N \rightarrow \infty } \E\Bigg[ \prod_{ d = 1 }^k \chi_d^* \Bigg] &\leq n^{ n^2 k } \binom{ n }{ 2 }^{ n k } \sum_{ \alpha_1 = 1 }^{ | \eta_0 | - | \eta_1 | } \ldots \sum_{ \alpha_k = 1 }^{ | \eta_{ k - 1 } | - | \eta_k | } \sum_{ ( \lambda_1, \mu_1 ) \in \Pi_2( [ \alpha_1 ] ) } \ldots \nonumber \\
&\sum_{ ( \lambda_k, \mu_k ) \in \Pi_2( [ \alpha_k ] ) } \lim_{ N \rightarrow \infty } \E\Bigg[ \osum_{ s_1^{ ( 1 ) } < \ldots < s_{ \alpha_1 }^{ ( 1 ) } = \tau_N( t_0 ) + 1 }^{ \tau_N( t_1 ) } \ldots \nonumber \\
&\osum_{ s_1^{ ( k ) } < \ldots < s_{ \alpha_k }^{ ( k ) } = \tau_N( t_{ k - 1 } ) + 1 }^{ \tau_N( t_k ) }\prod_{ d = 1 }^k  \Bigg\{ \prod_{ r \in \lambda_d } c_N( s_r^{ ( d ) } ) \Bigg\} \prod_{ r \in \mu_d } D_N( s_r^{ ( d ) } ) \Bigg]; \label{multi_merger_ub}
\end{align}
here, $n^{ n^2 k } \binom{ n }{ 2 }^{ n k }$ bounds $ \prod_{ d = 1 }^k \binom{ n }{ 2 }^{ | \lambda_d | } n^{ | \mu_d | n } \binom{ n - 2 }{ 2 }^{ | \mu_d | }$ using $| \lambda_d | + | \mu_d | \leq n$.
Transitions in which any $\alpha_d = 0$ result in a similar bound in which the corresponding factors on the right-hand side of \eqref{multi_merger_ub} are replaced by 1. Next, a multinomial expansion in reverse gives
\begin{multline}
\osum_{ s_1 < \ldots < s_{ \alpha } = \tau_N( t_{ d - 1 } ) + 1 }^{ \tau_N( t_d ) } \prod_{ r = 1 }^{ \alpha } c_N( s_r ) = \frac{ 1 }{ \alpha ! } \osum_{ \substack{ s_1 \neq \ldots \neq s_{ \alpha } = \tau_N( t_{ d - 1 } ) + 1 \\ \text{all distinct} } }^{ \tau_N( t_d ) } \prod_{ r = 1 }^{ \alpha } c_N( s_r )\\ 
\geq \frac{ 1 }{ \alpha ! } \Bigg( \sum_{ t = \tau_N( t_{ d - 1 } ) + 1 }^{ \tau_N( t_d ) } c_N( t ) \Bigg)^{ \alpha } - \frac{ 1 }{ \alpha ! } \binom{ \alpha }{ 2 } \Bigg( \sum_{ t = \tau_N( t_{ d - 1 } ) + 1 }^{ \tau_N( t_d ) } c_N( t )^2 \Bigg)\\ 
\times \Bigg( \sum_{ t = \tau_N( t_{ d - 1 } ) + 1 }^{ \tau_N( t_d ) } c_N( t ) \Bigg)^{ \alpha - 2 }, \label{time_expansion}
\end{multline}
where we take $\binom{0}{2} = \binom{1}{2} = 0$.
By definition of $\tau_N( t )$,
\begin{multline}
\frac{ 1 }{ \alpha ! } \osum_{ \substack{ s_1 \neq \ldots \neq s_{ \alpha } = \tau_N( t_{ d - 1 } ) + 1 \\ \text{all distinct} } }^{ \tau_N( t_d ) } \prod_{ r = 1 }^{ \alpha } c_N( s_r ) \leq \frac{ 1 }{ \alpha ! } \osum_{ s_1, \ldots,  s_{ \alpha } = \tau_N( t_{ d - 1 } ) + 1 }^{ \tau_N( t_d ) } \prod_{ r = 1 }^{ \alpha } c_N( s_r ) \\
\leq \frac{ [ t_d - t_{ d - 1 } + c_N( \tau_N( t_d ) ) ]^{ \alpha } }{ \alpha ! } \leq \frac{ ( t_d - t_{ d - 1 } + 1 )^{ \alpha } }{ \alpha ! }, \label{time_ub}
\end{multline}
because $c_N( t ) \leq 1$.
Suppose that $| \mu_i | > 0$ in \eqref{multi_merger_ub} for some $i \in [ k ]$.
Using $D_N( s_r^{ ( d ) } ) \leq c_N( s_r^{ ( d ) } )$, which is clear from 
\begin{equation*}
\frac{ 1 }{ N } \Bigg( \nu_t^{ ( i ) } + \frac{ 1 }{ N } \sum_{ j \neq i }^N ( \nu_t^{ ( j ) } )^2 \Bigg) \leq 1,
\end{equation*}
for all but one $D_N( s_r^{ ( d ) } )$-factor in \eqref{multi_merger_ub}, and substituting in \eqref{time_ub}, gives
\begin{align*}
\lim_{ N \rightarrow \infty } \E\Bigg[ \prod_{ d = 1 }^k \chi_d^* \Bigg] \leq {}& n^{ n^2 k } \binom{ n }{ 2 }^{ n k } \sum_{ \alpha_1 = 1 }^{ | \eta_0 | - | \eta_1 | } \ldots \sum_{ \alpha_k = 1 }^{ | \eta_{ k - 1 } | - | \eta_k | } \sum_{ ( \lambda_1, \mu_1 ) \in \Pi_2( [ \alpha_1 ] ) } \ldots \\
&\sum_{ ( \lambda_k, \mu_k ) \in \Pi_2( [ \alpha_k ] ) } \Bigg\{ \prod_{ d = 1 }^k \frac{ ( t_d - t_{ d - 1 } + 1 )^{ | \alpha_d | - \delta_{ d i } } }{ ( \alpha_d - \delta_{ d i } ) ! } \Bigg\}  \lim_{ N \rightarrow \infty } \E\Bigg[ \sum_{ s = \tau_N( t_{ i - 1 } ) + 1 }^{ \tau_N( t_i ) } D_N( s ) \Bigg],
\end{align*}
where $\delta_{ d i } = \mathds{ 1 }_{ \{ d = i \} }$ is the Kronecker delta. 
As for \eqref{multi_merger_ub}, cases for which some $\alpha_d$ are 0 can be handled by a straightforward modification.
All sums outside the expectation consist of a bounded number of terms in $N$, so \eqref{big_merger_bound} guarantees that the contribution of paths with $\sum_{ d = 1 }^k | \mu_d | > 0$ vanishes in the limit and only isolated, binary mergers can take place.

To describe transitions in which the only mergers that occur are isolated, binary mergers, we define $\tilde{ Q } = ( \tilde{ q }_{ \xi \eta } )_{ \xi \eta }$ to be the matrix obtained from $Q$ by setting its diagonal entries to 0.
Note that $( \tilde{ Q }^{ \alpha } )_{ \xi \eta }$ is precisely the number of ways of going from $\xi$ to $\eta$ in exactly $\alpha$ steps, where a step consists of merging a pair of blocks.

Now consider a transition from $\eta_{ d - 1 }$ to $\eta_d$ at respective times $\tau_N( t_{ d - 1 } )$ and $\tau_N( t_d )$ via binary mergers, that is, with $\alpha = | \eta_{ d - 1 } | - | \eta_d |$ and $\lambda = [ \alpha ]$.
By Lemma \ref{kingman_bounds}, Cases 1 and 2, its conditional probability is bounded by
\begin{multline*}
\chi_d \leq ( \tilde{ Q }^{ \alpha } )_{ \eta_{ d - 1 } \eta_d } ( 1 + O( N^{ -1 } ) )^{ \alpha } \osum_{ s_1 < \ldots < s_{ \alpha } = \tau_N( t_{ d - 1 } ) + 1 }^{ \tau_N( t_d ) } \Bigg\{ \prod_{ r = 1 }^{ \alpha } c_N( s_r ) \Bigg\} \\
\times \prod_{ \substack{ r = \tau_N( t_{ d - 1 } ) + 1 \\ r \neq s_1, \ldots, r \neq s_{ \alpha } } }^{ \tau_N( t_d ) } \Bigg[ 1 - \binom{ | \eta_{ d - 1 } | - | \{ i : s_i < r \} | }{ 2 } \{ 1 + O( N^{ -1 } ) \} \\
\times \Bigg( c_N( r ) - \binom{ n -1 }{ 2 } D_N( r ) \Bigg) \mathds{ 1 }_{ \{ c_N( r ) - \binom{ n - 1 }{ 2 } D_N( r ) < \binom{ n }{ 2 }^{ -1 } \} } \Bigg],
\end{multline*}
where $\chi_d$ is the restriction of $\chi_d^*$ to trajectories involving only isolated binary mergers, and the indicator functions ensure the bound remains valid for large values of $c_N( r ) - \binom{ n - 1 }{ 2 } D_N( r )$, which might otherwise cause the right hand side to be negative.
An expansion of the product on the second and third lines gives
\begin{align*}
&\prod_{ \substack{ r = \tau_N( t_{ d - 1 } ) + 1 \\ r \neq s_1, \ldots, r \neq s_{ \alpha } } }^{ \tau_N( t_d ) } \Bigg[ 1 - \binom{ | \eta_{ d - 1 } | - | \{ i : s_i < r \} | }{ 2 } \{ 1 + O( N^{ -1 } ) \} \\
&\phantom{\prod_{ r = \tau_N( t_{ d - 1 } ) + 1 }^{ \tau_N( t_d ) } \Bigg[} \times \Bigg( c_N( r ) - \binom{ n -1 }{ 2 } D_N( r ) \Bigg) \mathds{ 1 }_{ \{ c_N( r ) - \binom{ n - 1 }{ 2 } D_N( r ) < \binom{ n }{ 2 }^{ -1 } \} } \Bigg] \\
&= \sum_{ \beta = 0 }^{ \tau_N( t_d ) - \tau_N( t_{ d - 1 } ) - \alpha } ( -1 )^{ \beta } ( 1 + O( N^{ -1 } ) )^{ \beta } \osum_{ \substack{ r_1 < \ldots < r_{ \beta } = \tau_N( t_{ d - 1 } ) + 1 \\ \forall i : r_i \neq s_1, \ldots, r_i \neq s_{ \alpha } } }^{ \tau_N( t_d ) } \\
&\phantom{=} \prod_{ j = 1 }^{ \beta } \Bigg[ \binom{ | \eta_{ d - 1 } | - | \{ i : s_i < r_j \} | }{ 2 } \Bigg( c_N( r_j ) - \binom{ n -1 }{ 2 } D_N( r_j ) \Bigg) \mathds{ 1 }_{ \{ c_N( r_j ) - \binom{ n - 1 }{ 2 } D_N( r_j ) < \binom{ n }{ 2 }^{ -1 } \} } \Bigg].
\end{align*}
The product of binomial coefficients depends only on the pattern of orderings between times denoted by $\{ s_i \}_{ i \in [ \alpha ] }$ and $\{ r_j \}_{ j \in [ \beta ] }$, but is otherwise independent of the exact values of the time points.
Hence we have the bound
\begin{align}
\chi_d \leq {}& \sum_{ \beta = 0 }^{ \tau_N( t_d ) - \tau_N( t_{ d - 1 } ) - \alpha } ( 1 + O( N^{ -1 } ) )^{ \alpha + \beta } ( \tilde{ Q }^{ \alpha } )_{ \eta_{ d - 1 } \eta_d } \nonumber \\
&{}\times \sum_{ ( \lambda, \mu ) \in \Pi_2( [ \alpha + \beta ] ) : | \lambda | = \alpha } \Bigg\{ \prod_{ r \in \mu } - \binom{ | \eta_{ d - 1 } | - | \{ i \in \lambda : i < r \} | }{ 2 } \Bigg\} \nonumber \\
&{}\times \osum_{ s_1 < \ldots < s_{ \alpha + \beta } = \tau_N( t_{ d - 1 } ) + 1 }^{ \tau_N( t_d ) } \prod_{ r = 1 }^{ \alpha + \beta } \Bigg[ \Bigg( c_N( s_r ) - \binom{ n -1 }{ 2 } D_N( s_r ) \Bigg) \nonumber \\
&{} \phantom{\osum_{ s_1 < \ldots < s_{ \alpha + \beta } = \tau_N( t_{ d - 1 } ) + 1 }^{ \tau_N( t_d ) } \prod_{ r = 1 }^{ \alpha + \beta } \Bigg[} \times \mathds{ 1 }_{ \{ c_N( s_r ) - \binom{ n - 1 }{ 2 } D_N( s_r ) < \binom{ n }{ 2 }^{ -1 } \} } \Bigg]  \label{binary_ub_1}
\end{align}
Now, we can see that
\begin{equation}\label{path_counting}
( \tilde{ Q }^{ \alpha } )_{ \eta_{ d - 1 } \eta_d } \sum_{ ( \lambda, \mu ) \in \Pi_2( [ \alpha + \beta ] ) : | \lambda | = \alpha } \Bigg\{ \prod_{ r \in \mu } - \binom{ | \eta_{ d - 1 } | - | \{ i \in \lambda : i < r \} | }{ 2 } \Bigg\} = ( Q^{ \alpha + \beta } )_{ \eta_{ d - 1 } \eta_d } 
\end{equation}
by noting that
\begin{equation*}
\sum_{ ( \lambda, \mu ) \in \Pi_2( [ \alpha + \beta ] ) : | \lambda | = \alpha } ( \tilde{ Q }^{ \alpha } )_{ \eta_{ d - 1 } \eta_d } = [ ( \mathbb{ I } + \tilde{ Q } )^{ \alpha + \beta } ]_{ \eta_{ d - 1 } \eta_d }
\end{equation*}
is the number of discrete time paths of length $\alpha + \beta$ from $\eta_{ d - 1 }$ to $\eta_d$ using $\alpha$ binary mergers and $\beta$ identity steps, where $\mathbb{ I }$ is the identity matrix of the same size as $\tilde{ Q }$.
In \eqref{path_counting}, each identity step results in multiplication by the corresponding diagonal entry of $Q$, which justifies the equality.
Substituting \eqref{path_counting} into \eqref{binary_ub_1} gives
\begin{align*}
\chi_d \leq {}& \sum_{ \beta = 0 }^{ \tau_N( t_d ) - \tau_N( t_{ d - 1 } ) - \alpha }  ( 1 + O( N^{ -1 } ) )^{ \alpha + \beta }  ( Q^{ \alpha + \beta } )_{ \eta_{ d - 1 } \eta_d } \nonumber \\
&{}\times \osum_{ s_1 < \ldots < s_{ \alpha + \beta } = \tau_N( t_{ d - 1 } ) + 1 }^{ \tau_N( t_d ) } \prod_{ r = 1 }^{ \alpha + \beta } \Bigg[ \Bigg( c_N( s_r ) - \binom{ n -1 }{ 2 } D_N( s_r ) \Bigg)  \mathds{ 1 }_{ \{ c_N( s_r ) - \binom{ n - 1 }{ 2 } D_N( s_r ) < \binom{ n }{ 2 }^{ -1 } \} } \Bigg].
\end{align*}
Taking a product of $k$ transition probabilities with $\alpha_d := | \eta_d | - | \eta_{ d - 1 } |$ for $d \in [ k ]$, and expanding that product, gives the bound 
\begin{align}
\lim_{ N \rightarrow \infty } \E\Bigg[ \prod_{ d = 1 }^k \chi_d \Bigg] \leq {}& \lim_{ N \rightarrow \infty } \E\Bigg[ \sum_{ \beta_1 = 0 }^{ \infty } \ldots \sum_{ \beta_k = 0 }^{ \infty } ( 1 + O( N^{ -1 } ) )^{ | \bm{ \alpha } | + | \bm{ \beta } | } \nonumber \\
&\times \Bigg\{ \prod_{ d = 1 }^k ( Q^{ \alpha_d + \beta_d } )_{ \eta_{ d - 1 } \eta_d } \mathds{ 1 }_{ \{ \tau_N( t_d ) - \tau_N( t_{ d - 1 } ) \geq \alpha_d + \beta_d \} }\Bigg\} \nonumber \\
&\times \osum_{ s_1^{ ( 1 ) } < \ldots < s_{ \alpha_1 + \beta_1 }^{ ( 1 ) } = \tau_N( t_0 ) + 1 }^{ \tau_N( t_1 ) } \ldots \osum_{ s_1^{ ( k ) } < \ldots < s_{ \alpha_k + \beta_k }^{ ( k ) } = \tau_N( t_{ k - 1 } ) + 1 }^{ \tau_N( t_k ) } \nonumber \\ 
&\phantom{\times}\prod_{ d = 1 }^k \prod_{ r = 1 }^{ \alpha_d + \beta_d } \Bigg\{ \Bigg( c_N( s_r^{ ( d ) } ) - \binom{ n -1 }{ 2 } D_N( s_r^{ ( d ) } ) \Bigg) \nonumber \\
&\phantom{\times \prod_{ d = 1 }^k \prod_{ r = 1 }^{ \alpha_d + \beta_d }}\times \mathds{ 1 }_{ \{ c_N( s_r^{ ( d ) } ) - \binom{ n - 1 }{ 2 } D_N( s_r^{ ( d ) } ) < \binom{ n }{ 2 }^{ -1 } \} } \Bigg\} \Bigg]. \label{binary_ub_2}
\end{align}
We show in the Appendix that the hypotheses of the Fubini and dominated convergence theorems are satisfied, so that the expectation and limit can be passed inside the $k$-fold infinite summation over $\bm{ \beta }$.
That leaves
\begin{align}
\lim_{ N \rightarrow \infty } &\E\Bigg[ \prod_{ d = 1 }^k \chi_d \Bigg] \leq \sum_{ \beta_1 = 0 }^{ \infty } \ldots \sum_{ \beta_k = 0 }^{ \infty } \Bigg\{ \prod_{ d = 1 }^k ( Q^{ \alpha_d + \beta_d } )_{ \eta_{ d - 1 } \eta_d } \Bigg\} \nonumber \\
&\times \lim_{ N \rightarrow \infty } \E\Bigg[ \Bigg\{ \prod_{ d = 1 }^k \mathds{ 1 }_{ \{ \tau_N( t_d ) - \tau_N( t_{ d - 1 } ) \geq \alpha_d + \beta_d \} } \Bigg\} \osum_{ s_1^{ ( 1 ) } < \ldots < s_{ \alpha_1 + \beta_1 }^{ ( 1 ) } = \tau_N( t_0 ) + 1 }^{ \tau_N( t_1 ) } \ldots \nonumber \\
& \osum_{ s_1^{ ( k ) } < \ldots < s_{ \alpha_k + \beta_k }^{ ( k ) } = \tau_N( t_{ k - 1 } ) + 1 }^{ \tau_N( t_k ) } \prod_{ d = 1 }^k \prod_{ r = 1 }^{ \alpha_d + \beta_d } \Bigg\{ \Bigg( c_N( s_r^{ ( d ) } ) - \binom{ n -1 }{ 2 } D_N( s_r^{ ( d ) } ) \Bigg) \nonumber \\
&\phantom{\osum_{ s_1^{ ( k ) } < \ldots < s_{ \alpha_k + \beta_k }^{ ( k ) } = \tau_N( t_{ k - 1 } ) + 1 }^{ \tau_N( t_k ) } \prod_{ r = 1 }^{ \alpha_d + \beta_d }} \times \mathds{ 1 }_{ \{ c_N( s_r^{ ( d ) } ) - \binom{ n - 1 }{ 2 } D_N( s_r^{ ( d ) } ) < \binom{ n }{ 2 }^{ -1 } \} } \Bigg\} \Bigg]. \label{binary_ub_3}
\end{align}
One further expansion of the form
\begin{align*}
&\prod_{ d = 1 }^k \prod_{ r = 1 }^{ \alpha_d + \beta_d } \Bigg( c_N( s_r^{ ( d ) } ) - \binom{ n -1 }{ 2 } D_N( s_r^{ ( d ) } ) \Bigg) \\
&= \sum_{ ( \lambda_1, \mu_1 ) \in \Pi_2( [ \alpha_1 + \beta_1 ] ) } \ldots \sum_{ ( \lambda_k, \mu_k ) \in \Pi_2( [ \alpha_k + \beta_k ] ) } \Bigg( - \binom{ n - 1 }{ 2 } \Bigg)^{ \sum_{ d = 1 }^k | \mu_d | } \\
&\phantom{= \sum_{ ( \lambda_1, \mu_1 ) \in \Pi_2 } } \times \prod_{ d = 1 }^k \Bigg\{ \prod_{ r \in \mu_d } D_N( s_r^{ ( d ) } ) \Bigg\} \Bigg\{ \prod_{ r \in \lambda_d } c_N( s_r^{ ( d ) } ) \Bigg\}  \prod_{ r = 1 }^{ \alpha_d + \beta_d } \mathds{ 1 }_{ \{ c_N( s_r^{ ( d ) } ) - \binom{ n - 1 }{ 2 } D_N( s_r^{ ( d ) } ) < \binom{ n }{ 2 }^{ -1 } \} },
\end{align*}
and the argument used to show that multiple mergers vanish in the limit demonstrates that we can replace every $[ c_N( s_r^{ ( d ) } ) - \binom{ n - 1 }{ 2 } D_N( s_r^{ ( d ) } ) ]$ with $c_N( s_r^{ ( d ) } )$ without affecting the limit.

Consider now a generic term in \eqref{binary_ub_3} for which $\prod_{ d = 1 }^k ( Q^{ \alpha_d + \beta_d } )_{ \eta_{ d - 1 } \eta_d }$ is positive, which requires $| \bm{ \beta } |$ to be even by \eqref{path_counting}.
We use the penultimate inequality in \eqref{time_ub}, and expand the resulting product to bound such the expectations in \eqref{binary_ub_3} by
\begin{align}
&\lim_{ N \rightarrow \infty } \E\Bigg[ \Bigg\{ \prod_{ d = 1 }^k \mathds{ 1 }_{ \{ \tau_N( t_d ) - \tau_N( t_{ d - 1 } ) \geq \alpha_d + \beta_d \} } \Bigg\}  \osum_{ s_1^{ ( 1 ) } < \ldots < s_{ \alpha_1 + \beta_1 }^{ ( 1 ) } = \tau_N( t_0 ) + 1 }^{ \tau_N( t_1 ) } \ldots \osum_{ s_1^{ ( k ) } < \ldots < s_{ \alpha_k + \beta_k }^{ ( k ) } = \tau_N( t_{ k - 1 } ) + 1 }^{ \tau_N( t_k ) } \nonumber \\
&\phantom{\lim_{ N \rightarrow \infty } \E\Bigg[} \prod_{ d = 1 }^k \prod_{ r = 1 }^{ \alpha_d + \beta_d } c_N( s_r^{ ( d ) } ) \mathds{ 1 }_{\{ c_N( s_r^{ ( d ) } ) - \binom{ n - 1 }{ 2 } D_N( s_r^{ ( d ) } ) < \binom{ n }{ 2 }^{ -1 } \}} \Bigg] \nonumber \\
&\leq \sum_{ j_1 = 0 }^{ \alpha_1 } \ldots \sum_{ j_k = 0 }^{ \alpha_k } \prod_{ d = 1 }^k \binom{ \alpha_d + \beta_d }{ j_d } \frac{ ( t_d - t_{ d - 1 } )^{ \alpha_d + \beta_d - j_d } }{ ( \alpha_d + \beta_d ) ! }  \lim_{ N \rightarrow \infty } \E[ c_N( \tau_N( t_d ) )^{ j_d } ] \nonumber \\
&= \prod_{ d = 1 }^k \frac{ ( t_d - t_{ d - 1 } )^{ \alpha_d + \beta_d } }{ ( \alpha_d + \beta_d ) ! }, \label{positive_terms_expansion}
\end{align}
where the last step follows from $c_N( t ) \leq 1$ and Lemma \ref{stopping_time_lemma} in the Appendix.
Expansion \eqref{time_expansion} also yields the lower bound
\begin{align*}
\osum_{ s_1 < \ldots < s_{ \alpha } = \tau_N( t_{ d - 1 } ) + 1 }^{ \tau_N( t_d ) } \prod_{ r = 1 }^{ \alpha } c_N( s_r ) \geq {}& \frac{ \{ t_d - t_{ d - 1 } - c_N( \tau_N( t_{ d - 1 } ) ) \}^{ \alpha } }{ \alpha ! } \\
&- \binom{ \alpha }{ 2 } \Bigg( \sum_{ t = \tau_N( t_{ d - 1 } ) + 1 }^{ \tau_N( t_d ) } c_N( t )^2 \Bigg) \frac{ ( t_d - t_{ d - 1 } + 1 )^{ \alpha - 2 } }{ \alpha ! }, 
\end{align*}
meaning that expectations in \eqref{binary_ub_3} with odd $| \bm{ \beta } |$ can be lower bounded by
\begin{align*}
&\lim_{ N \rightarrow \infty } \E\Bigg[ \Bigg\{ \prod_{ d = 1 }^k \mathds{ 1 }_{ \{ \tau_N( t_d ) - \tau_N( t_{ d - 1 } ) \geq \alpha_d + \beta_d \} } \Bigg\}  \osum_{ s_1^{ ( 1 ) } < \ldots < s_{ \alpha_1 + \beta_1 }^{ ( 1 ) } = \tau_N( t_0 ) + 1 }^{ \tau_N( t_1 ) } \ldots \osum_{ s_1^{ ( k ) } < \ldots < s_{ \alpha_k + \beta_k }^{ ( k ) } = \tau_N( t_{ k - 1 } ) + 1 }^{ \tau_N( t_k ) } \nonumber \\
&\phantom{\lim_{ N \rightarrow \infty } \E\Bigg[} \prod_{ d = 1 }^k \prod_{ r = 1 }^{ \alpha_d + \beta_d } c_N( s_r^{ ( d ) } ) \mathds{ 1 }_{ \{ c_N( s_r^{ ( d ) } ) - \binom{ n - 1 }{ 2 } D_N( s_r^{ ( d ) } ) < \binom{ n }{ 2 }^{ -1 } \} } \Bigg] \nonumber \\
&\geq \sum_{ ( \lambda, \mu ) \in \Pi_2( [ k ] ) } \E\Bigg[  \Bigg\{ \prod_{ d = 1 }^k \mathds{ 1 }_{ \{ \tau_N( t_d ) - \tau_N( t_{ d - 1 } ) \geq \alpha_d + \beta_d \} } \Bigg\}  \Bigg\{ \prod_{ d = \tau_N( t_0 ) + 1 }^{ \tau_N( t_k ) } \mathds{ 1 }_{\{ c_N( d ) - \binom{ n - 1 }{ 2 } D_N( d ) < \binom{ n }{ 2 }^{ -1 } \}} \Bigg\} \\
&\phantom{\geq}\times  \Bigg\{ \prod_{ d \in \lambda } \frac{ \{ t_d - t_{ d - 1 } - c_N( \tau_N( t_{ d - 1 } ) ) \}^{ \alpha_d + \beta_d } }{ ( \alpha_d + \beta_d ) ! } \Bigg\} \\
&\phantom{\geq} \times \prod_{ d \in \mu } - \binom{ \alpha_d + \beta_d }{ 2 } \Bigg( \sum_{ t = \tau_N( t_{ d - 1 } ) + 1 }^{ \tau_N( t_d ) } c_N( t )^2 \Bigg) \frac{ ( t_d - t_{ d - 1 } + 1 )^{ \alpha_d + \beta_d - 2 } }{ ( \alpha_d + \beta_d ) ! } \Bigg].
\end{align*}
Using the lower bound for terms with a positive sign and upper bound for terms with a negative sign from
\begin{equation*}
0 \leq \sum_{ t = \tau_N( t_{ d - 1 } ) + 1 }^{ \tau_N( t_d ) } c_N( t )^2 \leq \sum_{ t = \tau_N( t_{ d - 1 } ) + 1 }^{ \tau_N( t_d ) } c_N( t ) \leq ( t_d - t_{ d - 1 } + 1 )
\end{equation*}
on all factors $d \in \mu$ apart from one, and then \eqref{binary_bound_2}, shows that terms with $| \mu | > 0$ vanish in the limit, after which an expansion akin to \eqref{positive_terms_expansion} results in
\begin{align}
&\lim_{ N \rightarrow \infty } \E\Bigg[ \Bigg\{ \prod_{ d = 1 }^k \mathds{ 1 }_{ \{ \tau_N( t_d ) - \tau_N( t_{ d - 1 } ) \geq \alpha_d + \beta_d \} } \Bigg\} \Bigg\{ \prod_{ d = \tau_N( t_0 ) + 1 }^{ \tau_N( t_k ) } \mathds{ 1 }_{\{ c_N( d ) - \binom{ n - 1 }{ 2 } D_N( d ) < \binom{ n }{ 2 }^{ -1 } \} } \Bigg\} \nonumber \\
&\phantom{\lim_{ N \rightarrow \infty } \E\Bigg[}\times  \osum_{ s_1^{ ( 1 ) } < \ldots < s_{ \alpha_1 + \beta_1 }^{ ( 1 ) } = \tau_N( t_0 ) + 1 }^{ \tau_N( t_1 ) } \ldots \osum_{ s_1^{ ( k ) } < \ldots < s_{ \alpha_k + \beta_k }^{ ( k ) } = \tau_N( t_{ k - 1 } ) + 1 }^{ \tau_N( t_k ) }  \prod_{ d = 1 }^k \prod_{ r = 1 }^{ \alpha_d + \beta_d } c_N( s_r^{ ( d ) } ) \Bigg] \nonumber \\
\geq {}& \Bigg\{ \prod_{ d = 1 }^k \frac{ ( t_d - t_{ d - 1 } )^{ \alpha_d + \beta_d } }{ ( \alpha_d + \beta_d ) ! } \Bigg\} \lim_{ N \rightarrow \infty }\E\Bigg[ \Bigg\{ \prod_{ d = 1 }^k \mathds{ 1 }_{ \{ \tau_N( t_d ) - \tau_N( t_{ d - 1 } ) \geq \alpha_d + \beta_d \} } \Bigg\} \nonumber \\
&\phantom{\lim_{ N \rightarrow \infty } \E\Bigg[} \times \Bigg\{ \prod_{ d = \tau_N( t_0 ) + 1 }^{ \tau_N( t_k ) } \mathds{ 1 }_{\{ c_N( d ) - \binom{ n - 1 }{ 2 } D_N( d ) < \binom{ n }{ 2 }^{ -1 } \} } \Bigg\} \Bigg]. \label{negative_terms_expansion}
\end{align}
Now, for $0 \leq s < t < \infty$ and $\alpha \in \mathbb{ N }$, we have
\begin{align*}
\{ \tau_N( t ) - \tau_N( s ) < \alpha \} &\subseteq \Bigg\{ \min\Bigg\{ p \geq 1 : \sum_{ r = \tau_N( s ) + 1 }^{ \tau_N( s ) + p } c_N( r ) \geq t - s \Bigg\} < \alpha \Bigg\} \\
&= \Bigg\{ \min\Bigg\{ p \geq 1 : \sum_{ r = 1 }^p c_N( \tau_N( s ) + r ) \geq t - s \Bigg\} < \alpha \Bigg\} \\
&\subseteq \Bigg\{ \sum_{ r = 1 }^{ \alpha } c_N( \tau_N( s ) + r ) \geq t - s \Bigg\},
\end{align*}
and hence by De Morgan's law and the union bound,
\begin{align}
&\E\Bigg[ \Bigg\{ \prod_{ d = 1 }^k \mathds{ 1 }_{ \{ \tau_N( t_d ) - \tau_N( t_{ d - 1 } ) \geq \alpha_d \} } \Bigg\} \prod_{ d = \tau_N( t_0 ) + 1 }^{ \tau_N( t_k ) } \mathds{ 1 }_{\{ c_N( d ) - \binom{ n - 1 }{ 2 } D_N( d ) < \binom{ n }{ 2 }^{ -1 } \} } \Bigg] \nonumber \\
&\geq 1 - \sum_{ d = 1 }^k \bbP\Bigg( \sum_{ r = 1 }^{ \alpha_d } c_N( \tau_N( t_{ d - 1 } ) + r ) \geq t_d - t_{ d - 1 } \Bigg) -  \mathbb{E}\Bigg[ \sum_{ d = \tau_N( t_0 ) + 1 }^{ \tau_N( t_k ) }\mathds{ 1 }_{ \{ c_N( d ) - \binom{ n - 1 }{ 2 } D_N( d ) \geq \binom{ n }{ 2 }^{ -1 } \} } \Bigg]\ \nonumber \\
&\geq 1 - \sum_{ d = 1 }^k \frac{ 1 }{ t_d - t_{ d - 1 } } \sum_{ r = 1 }^{ \alpha_d } \E[ c_N( \tau_N( t_{ d - 1 } + r ) ) ] \nonumber \\
&\phantom{\geq 1 } {}- \binom{n}{2}^2 \E\Bigg[ \sum_{ d = \tau_N( t_0 ) + 1 }^{ \tau_N( t_k ) } \Bigg( c_N( d ) - \binom{ n - 1 }{ 2 } D_N( d ) \Bigg)^2 \Bigg] \to 1 \label{indicator_lb}
\end{align}
as $N \to \infty$, where the second inequality follows by applying  the Markov inequality to the first sum, and two invocations of Lemma \ref{optional_stopping_lemma} with the conditional Markov inequality in between to the second sum.
Convergence follows from \eqref{big_merger_bound}, \eqref{binary_bound_2}, and Lemma \ref{stopping_time_lemma} in the Appendix.
Substituting \eqref{positive_terms_expansion} for terms in \eqref{binary_ub_3} with $| \bm{ \beta } |$ even, and \eqref{negative_terms_expansion} as well as \eqref{indicator_lb} for terms in \eqref{binary_ub_3} with $| \bm{ \beta } |$ odd, and using the fact that $( Q^j )_{ \eta_{ d - 1 } \eta_d } = 0$ for $j < \alpha_d$ gives
\begin{align*}
\lim_{ N \rightarrow \infty } \E\Bigg[ \prod_{ d = 1 }^k \chi_d \Bigg] &\leq \sum_{ \beta_1 = 0 }^{ \infty } \ldots \sum_{ \beta_k = 0 }^{ \infty } \Bigg\{ \prod_{ d = 1 }^k ( Q^{ \alpha_d + \beta_d } )_{ \eta_{ d - 1 } \eta_d } \frac{ ( t_d - t_{ d - 1 } )^{ \alpha_d + \beta_d } }{ ( \alpha_d + \beta_d ) ! }  \Bigg\} = \prod_{ d = 1 }^k ( e^{ Q ( t_d - t_{ d - 1 } ) } )_{ \eta_{ d - 1 } \eta_d },
\end{align*}
where the required absolute convergence for the final rearrangement is guaranteed by the calculation in the Appendix used to bound \eqref{binary_ub_2}.

In the other direction, Lemma \ref{kingman_bounds}, Cases 1 and 2 give the lower bound
\begin{align*}
\chi_d \geq {}& \osum_{ s_1 < \ldots < s_{ \alpha } = \tau_N( t_{ d - 1 } ) + 1 }^{ \tau_N( t_d ) } ( \tilde{ Q }^{ \alpha } )_{ \eta_{ d - 1 } \eta_d } \Bigg\{ \prod_{ r = 1 }^{ \alpha } \mathds{ 1 }_{ \{ c_N( s_r) > \binom{ n - 2 }{ 2 } D_N( s_r ) \} } \\
&\times \Bigg( c_N( s_r ) - \binom{ n - 2 }{ 2 } \{ 1+ O( N^{ -1 } ) \} D_N( s_r ) \Bigg) \Bigg\} \\
&\times \prod_{ \substack{ r = \tau_N( t_{ d - 1 } ) + 1 \\ r \neq s_1, \ldots, r \neq s_{ \alpha } } }^{ \tau_N( t_d ) } \Bigg\{ 1 - \binom{ n }{ 2 } \frac{ ( 3 n - 1 ) ( n - 2 ) }{ 6 N^2 } \{ 1 + O( N^{ -1 } ) \} \\
&\phantom{\times} - \binom{ | \eta_{ d - 1 } | - | \{ i : s_i < r \} | }{ 2 } \{ 1 + O( N^{ -1 } ) \} c_N( r ) \Bigg\} \mathds{1}_{ \{ c_N( r ) < \binom{ n }{ 2 }^{ -1 }  \} }.
\end{align*}
A multinomial expansion of the product spanning the last two lines yields 
\begin{align*}
\chi_d \geq {}& \Bigg\{ \prod_{ r = \tau_N( t_{ d - 1 } + 1 ) }^{ \tau_N( t_d ) } \mathds{1}_{ \{ c_N( r ) < \binom{ n }{ 2 }^{ -1 }  \} } \Bigg\} \sum_{ \beta = 0 }^{ \tau_N( t_d ) - \tau_N( t_{ d - 1 } ) - \alpha } \frac{ ( 1 + O( N^{ -1 } ) )^{ \beta } }{ ( 1 + O( N^{ -2 } ) )^{ \alpha + \beta } } \\
&\times ( \tilde{ Q }^{ \alpha } )_{ \eta_{ d - 1 } \eta_d }  \sum_{ ( \lambda, \mu ) \in \Pi_2( [ \alpha + \beta ] ) : | \lambda | = \alpha } \Bigg\{ \prod_{ r \in \mu } - \binom{ | \eta_{ d - 1 } | - | \{ i \in \lambda : i < r \} | }{ 2 } \Bigg\} \\
&\times \Bigg( 1 - \binom{ n }{ 2 }^2 \frac{ 1 }{ N^2 } \Bigg)^{ \tau_N( t_d ) - \tau_N( t_{ d - 1 } ) } \osum_{ s_1 < \ldots < s_{ \alpha + \beta } = \tau_N( t_{ d - 1 } ) + 1 }^{ \tau_N( t_d ) } \Bigg\{ \prod_{ r \in \mu } c_N( s_r ) \Bigg\} \\
&\times \Bigg\{ \prod_{ r \in \lambda } \mathds{ 1 }_{ \{ c_N( s_r) > \binom{ n - 2 }{ 2 } D_N( s_r ) \} }  \Bigg( c_N( s_r ) - \binom{ n - 2 }{ 2 } \{ 1 + O( N^{ -1 } ) \} D_N( s_r ) \Bigg) \Bigg\} , 
\end{align*}
for sufficiently large $N$; and expanding the product over $\lambda$ gives
\begin{align*}
\chi_d \geq {}& \Bigg\{ \prod_{ r = \tau_N( t_{ d - 1 } ) + 1 }^{ \tau_N( t_d ) } \mathds{1}_{ \{ c_N( r ) < \binom{ n }{ 2 }^{ -1 }  \} } \mathds{ 1 }_{ \{ c_N( r ) > \binom{ n - 2 }{ 2 } D_N( r ) \} } \Bigg\} \\
&\times \sum_{ \beta = 0 }^{ \tau_N( t_d ) - \tau_N( t_{ d - 1 } ) - \alpha } ( \tilde{ Q }^{ \alpha } )_{ \eta_{ d - 1 } \eta_d } \sum_{ ( \lambda, \mu, \pi ) \in \Pi_3( [ \alpha + \beta ] ) : | \mu | = \beta } \left\{-\binom{ n - 2 }{ 2 }\right\}^{ | \pi | }  \\
&\times  \frac{ ( 1 + O( N^{ -1 } ) )^{ \beta + | \pi | } }{ ( 1 + O( N^{ -2 } ) )^{ \alpha + \beta } } \Bigg( 1 - \binom{ n }{ 2 }^2 \frac{ 1 }{ N^2 } \Bigg)^{ \tau_N( t_d ) - \tau_N( t_{ d - 1 } ) } \\
&\times \Bigg\{ \prod_{ r \in \mu } - \binom{ | \eta_{ d - 1 } | - | \{ i \in \lambda \cup \pi : i < r \} | }{ 2 } \Bigg\} \\
&\times \osum_{ s_1 < \ldots < s_{ \alpha + \beta } = \tau_N( t_{ d - 1 } ) + 1 }^{ \tau_N( t_d ) } \Bigg\{ \prod_{ r \in \lambda \cup \mu } c_N( s_r ) \Bigg\} \prod_{ r \in \pi }  D_N( s_r ).
\end{align*}
Via a further multinomial expansion, the lower bound for the $k$-step transition probability can be written as
\begin{align}
\lim_{ N \rightarrow \infty } &\E\Bigg[ \prod_{ d = 1 }^k \chi_d \Bigg] \geq \lim_{ N \rightarrow \infty } \E\Bigg[ \Bigg( 1 - \binom{ n }{ 2 }^2 \frac{ 1 }{ N^2 } \Bigg)^{ \tau_N( t_k ) - \tau_N( t_0 ) } \nonumber \\
&\times \Bigg\{ \prod_{ r = \tau_N( t_0 ) + 1 }^{ \tau_N( t_d ) } \mathds{1}_{ \{ c_N( r ) < \binom{ n }{ 2 }^{ -1 }  \} } \mathds{ 1 }_{ \{ c_N( r ) > \binom{ n - 2 }{ 2 } D_N( r ) \} } \Bigg\} \sum_{ \beta_1 = 0 }^{ \infty } \ldots \nonumber \\
&\sum_{ \beta_k = 0 }^{ \infty } \sum_{ ( \lambda_1, \mu_1, \pi_1 ) \in \Pi_3( [ \alpha_1 + \beta_1 ] ) : | \mu_1 | = \beta_1 } \ldots \sum_{ ( \lambda_k, \mu_k, \pi_k ) \in \Pi_3( [ \alpha_k + \beta_k ] ) : | \mu_k | = \beta_k } \nonumber \\
&\times \left\{-\binom{ n - 2 }{ 2 }\right\}^{ \sum_{ d = 1 }^k | \pi_d | } \frac{ ( 1 + O( N^{ -1 } ) )^{ | \bm{ \beta } | + \sum_{ d = 1 }^k | \pi_d | } }{ ( 1 + O( N^{ -2 } ) )^{ | \bm{ \alpha } | + | \bm{ \beta } | } } \nonumber \\
&\times \Bigg\{ \prod_{ d = 1 }^k ( \tilde{ Q }^{ \alpha_d } )_{ \eta_{ d - 1 } \eta_d } \prod_{ r \in \mu_d } - \binom{ | \eta_{ d - 1 } | - | \{ i \in \lambda_d \cup \pi_d : i < r \} | }{ 2 } \Bigg\} \nonumber \\
&\times  \osum_{ s_1^{ ( 1 ) } < \ldots < s_{ \alpha_1 + \beta_1 }^{ ( 1 ) } = \tau_N( t_0 ) + 1 }^{ \tau_N( t_1 ) } \ldots \osum_{ s_1^{ ( k ) } < \ldots < s_{ \alpha_k + \beta_k }^{ ( k ) } = \tau_N( t_{ k - 1 } ) + 1 }^{ \tau_N( t_k ) } \nonumber \\ 
&\prod_{ d = 1 }^k \mathds{ 1 }_{ \{ \tau_N( t_d ) - \tau_N( t_{ d - 1 } ) \geq \alpha_d + \beta_d \} } \Bigg\{ \prod_{ r \in \lambda_d \cup \mu_d } c_N( s_r^{ ( d ) } ) \Bigg\} \prod_{ r \in \pi_d }  D_N( s_r^{ ( d ) } ) \Bigg]. \label{pre-fubini_lb}
\end{align}
As for the upper bound, we verify in the Appendix that passing the expectation and the limit through the infinite sums is justified, whereupon \eqref{big_merger_bound}, and the argument used to show that mergers involving more than two lineages cannot happen in the limit, implies that the contribution of terms with $\sum_{ d = 1 }^k | \pi_d | > 0$ vanishes in the limit.
Applying \eqref{path_counting} to the remaining terms gives
\begin{align*}
\lim_{ N \rightarrow \infty } \E\Bigg[ \prod_{ d = 1 }^k \chi_d \Bigg] \geq {}& \sum_{ \beta_1 = 0 }^{ \infty } \ldots \sum_{ \beta_k = 0 }^{ \infty }  \Bigg\{ \prod_{ d = 1 }^k ( Q^{ \alpha_d + \beta_d } )_{ \eta_{ d - 1 } \eta_d }  \Bigg\} \\
& \times \lim_{ N \rightarrow \infty } \E\Bigg[  \Bigg( 1 - \binom{ n }{ 2 }^2 \frac{ 1 }{ N^2 } \Bigg)^{ \tau_N( t_k ) - \tau_N( t_0 ) }   \Bigg\{ \prod_{ d = 1 }^k \mathds{ 1 }_{ \{ \tau_N( t_d ) - \tau_N( t_{ d - 1 } ) \geq \alpha_d + \beta_d \} } \Bigg\} \\
& \times \Bigg\{ \prod_{ r = \tau_N( t_0 ) + 1 }^{ \tau_N( t_d ) } \mathds{1}_{ \{ c_N( r ) < \binom{ n }{ 2 }^{ -1 }  \} } \mathds{ 1 }_{ \{ c_N( r ) > \binom{ n - 2 }{ 2 } D_N( r ) \} } \Bigg\} \\
&\times \osum_{ s_1^{ ( 1 ) } < \ldots < s_{ \alpha_1 + \beta_1 }^{ ( 1 ) } = \tau_N( t_0 ) + 1 }^{ \tau_N( t_1 ) } \ldots \osum_{ s_1^{ ( k ) } < \ldots < s_{ \alpha_k + \beta_k }^{ ( k ) } = \tau_N( t_{ k - 1 } ) + 1 }^{ \tau_N( t_k ) } \prod_{ d = 1 }^k \prod_{ r = 1 }^{ \alpha_d + \beta_d } c_N( s_r^{ ( d ) } ) \Bigg].
\end{align*}
Likewise, with appropriately modified indicator functions, the expansion \eqref{positive_terms_expansion} applied to terms with negative sign ($| \bm{ \beta } |$ odd), and \eqref{negative_terms_expansion} applied to terms with a positive sign ($| \bm{ \beta } |$ even), followed by Lemma \ref{stopping_time_lemma} in the Appendix, show that 
\begin{align*}
\lim_{ N \rightarrow \infty } \E\Bigg[ \prod_{ d = 1 }^k \chi_d \Bigg] \geq {}& \sum_{ \beta_1 = 0 }^{ \infty } \ldots \sum_{ \beta_k = 0 }^{ \infty }  \Bigg\{ \prod_{ d = 1 }^k ( Q^{ \alpha_d + \beta_d } )_{ \eta_{ d - 1 } \eta_d } \frac{ ( t_d - t_{ d - 1 } )^{ \alpha_d + \beta_d } }{ ( \alpha_d + \beta_d ) ! } \Bigg\} \\
&\times \lim_{ N \rightarrow \infty } \E\Bigg[  \Bigg( 1 - \binom{ n }{ 2 }^2 \frac{ 1 }{ N^2 } \Bigg)^{ \tau_N( t_k ) - \tau_N( t_0 ) } \Bigg\{ \prod_{ d = 1 }^k \mathds{ 1 }_{ \{ \tau_N( t_d ) - \tau_N( t_{ d - 1 } ) \geq \alpha_d + \beta_d \} } \Bigg\} \\
&\phantom{\times \lim_{ N \rightarrow \infty } \E\Bigg[} \times \Bigg\{ \prod_{ r = \tau_N( t_0 ) + 1 }^{ \tau_N( t_d ) } \mathds{1}_{ \{ c_N( r ) < \binom{ n }{ 2 }^{ -1 }  \} } \mathds{ 1 }_{ \{ c_N( r ) > \binom{ n - 2 }{ 2 } D_N( r ) \} } \Bigg\} \Bigg].
\end{align*}
For terms with even $| \bm{ \beta } |$, the bound $1 - x \leq e^{ -x }$ gives the inequality
\begin{align*}
&\E\Bigg[  \Bigg( 1 - \binom{ n }{ 2 }^2 \frac{ 1 }{ N^2 } \Bigg)^{ \tau_N( t_k ) - \tau_N( t_0 ) } \Bigg\{ \prod_{ d = 1 }^k \mathds{ 1 }_{ \{ \tau_N( t_d ) - \tau_N( t_{ d - 1 } ) \geq \alpha_d + \beta_d \} } \Bigg\} \\
&\phantom{\E\Bigg[} \times  \Bigg\{ \prod_{ r = \tau_N( t_0 ) + 1 }^{ \tau_N( t_d ) } \mathds{1}_{ \{ c_N( r ) < \binom{ n }{ 2 }^{ -1 }  \} } \mathds{ 1 }_{ \{ c_N( r ) > \binom{ n - 2 }{ 2 } D_N( r ) \} } \Bigg\} \Bigg] \\
 &\geq \E\Bigg[ \Bigg\{ \prod_{ d = 1 }^k \mathds{ 1 }_{ \{ \tau_N( t_d ) - \tau_N( t_{ d - 1 } ) \geq \alpha_d + \beta_d \} } \Bigg\}  \times \Bigg\{ \prod_{ r = \tau_N( t_0 ) + 1 }^{ \tau_N( t_d ) } \mathds{1}_{ \{ c_N( r ) < \binom{ n }{ 2 }^{ -1 }  \} } \mathds{ 1 }_{ \{ c_N( r ) > \binom{ n - 2 }{ 2 } D_N( r ) \} } \Bigg\}  \\
&\phantom{\geq \E\Bigg[}+ \log\Bigg( 1 - \binom{ n }{ 2 }^2 \frac{ 1 }{ N^2 } \Bigg) [ \tau_N( t_k ) - \tau_N( t_0 ) ] \Bigg].
\end{align*}
Assumption \eqref{tau_bound} ensures that the second term vanishes as $N \to \infty$, and we show in Lemma \ref{indicator_lemma} in the Appendix that the expectation of the product of indicators converges to 1.
Hence,
\begin{align*}
&\lim_{ N \rightarrow \infty } \E\Bigg[  \Bigg( 1 - \binom{ n }{ 2 }^2 \frac{ 1 }{ N^2 } \Bigg)^{ \tau_N( t_k ) - \tau_N( t_0 ) } \Bigg\{ \prod_{ d = 1 }^k \mathds{ 1 }_{ \{ \tau_N( t_d ) - \tau_N( t_{ d - 1 } ) \geq \alpha_d + \beta_d \} } \Bigg\} \\
&\phantom{\lim_{ N \rightarrow \infty } \E\Bigg[} \times  \Bigg\{ \prod_{ r = \tau_N( t_0 ) + 1 }^{ \tau_N( t_d ) } \mathds{1}_{ \{ c_N( r ) < \binom{ n }{ 2 }^{ -1 }  \} } \mathds{ 1 }_{ \{ c_N( r ) > \binom{ n - 2 }{ 2 } D_N( r ) \} } \Bigg\} \Bigg] \\
 &\geq 1 + \lim_{ N \rightarrow \infty } N C_{ t_k, t_0 } \log\Bigg( 1 - \binom{ n }{ 2 }^2 \frac{ 1 }{ N^2 } \Bigg) = 1.
\end{align*}
A corresponding upper bound of 1 for terms with odd $| \bm{ \beta } |$ is immediate, resulting in the overall lower bound
\begin{align*}
\lim_{ N \rightarrow \infty } \E\Bigg[ \prod_{ d = 1 }^k \chi_d \Bigg] &\geq \sum_{ \beta_1 = 0 }^{ \infty } \ldots \sum_{ \beta_k = 0 }^{ \infty }  \Bigg\{ \prod_{ d = 1 }^k ( Q^{ \alpha_d + \beta_d } )_{ \eta_{ d - 1 } \eta_d } \frac{ ( t_d - t_{ d - 1 } )^{ \alpha_d + \beta_d } }{ ( \alpha_d + \beta_d ) ! } \Bigg\} \\
&\geq \prod_{ d = 1 }^k ( e^{ Q ( t_d - t_{ d - 1 } ) } )_{ \eta_{ d - 1 } \eta_d },
\end{align*}
which again follows from the fact that $( Q^j )_{ \eta_{ d - 1 } \eta_d } = 0$ for $j < \alpha_d$, and the bound obtained for \eqref{pre-fubini_lb} in the Appendix.
\end{proof}

Our next aim is to show that particle filters with multinomial resampling can satisfy \eqref{big_merger_bound} -- \eqref{tau_bound}.
We require one further preparatory lemma.
\vskip 11pt
\begin{lem}\label{conditional_lemma}
Suppose the kernels $K_t( x, dx' )$ in Algorithm \ref{particle_filter} have respective densities $q_t( x, x' ) dx'$, and that
\begin{gather}\label{weight_bound}
\frac{ 1 }{ a } \leq g_t( x, x' ) \leq a,\\
\label{mixing_bound}
\varepsilon h( x' ) \leq q_t( x, x' ) \leq \frac{ 1 }{ \varepsilon } h( x' ),
\end{gather}
for some constants $0 < \varepsilon \leq 1 \leq a < \infty$, and probability density $h( x )$, uniformly in $t$ as well as both arguments.
Then SMC algorithms with multinomial resampling (i.e.,~Algorithm \ref{particle_filter}, for which the \texttt{Resample} method on line 6 is Algorithm \ref{multinomial_resampling}) satisfy
\begin{align}\label{D_N_bound}
\E[ D_N( t ) | \calF_{ t - 1 } ] &\leq \frac{ C_1 }{ N } \E[ c_N( t ) | \calF_{ t - 1 } ],\\
\label{c_N_2_bound}
\E[ c_N( t )^2 | \calF_{ t - 1 } ] &\leq \frac{ C_2 }{ N } \E[ c_N( t ) | \calF_{ t - 1 } ],
\end{align}
for constants $C_1, C_2 > 0$ that are independent of $N$, and 
\begin{equation}\label{c_N_bound}
\frac{ \varepsilon^4 }{ N a^4 } \leq \E[ c_N( t ) | \calF_{ t - 1 } ] \leq \frac{ a^4 }{ N \varepsilon^4 }.
\end{equation}
\end{lem}
\vskip 11pt
\begin{rmk}
Assumptions \eqref{weight_bound} and \eqref{mixing_bound} are strong, and can only be expected to hold on compact state spaces.
Many rigorous results about SMC require similarly strong assumptions, but are robust to violations of them in practice \cite{DelMoral01, Chopin04, Kuensch05, Cerou11, Jacob15}.
\end{rmk}
\begin{proof}[Proof of Lemma \ref{conditional_lemma}]
Recall that our reverse-time perspective results in SMC algorithms whose time steps progress backwards through time points $t + 1, t, t - 1, \ldots, 0$.
Thus, for any integrable function $f( \a_t )$, the forwards-in-time Markov property of SMC algorithms gives
\begin{align*}
\E[ f( \a_t ) | \calF_{ t - 1 } ] &= \E[ \E[ f( \a_t ) | \X_{ t - 1 }, \w_{ t - 1 } ] | \calF_{ t - 1 } ]\\
&= \E[ \E[ f( \a_t ) | \a_{ t + 1 }, \X_{ t + 1 }, \X_t, \X_{ t - 1 }, \w_{ t - 1 } ] | \calF_{ t - 1 } ].
\end{align*}
For multinomial resampling, the law with respect to which the inner conditional expectation is taken is
\begin{equation*}
\bbP( \a_t = \a | \a_{ t + 1 }, \X_{ t + 1 }, \X_t, \X_{ t - 1 }, \w_{ t - 1 } ) \propto \prod_{ i = 1 }^N g_t( X_{ t + 1 }^{ ( a_{ t + 1 }^{ ( a_i ) } ) }, X_t^{ ( a_i ) } ) q_{ t - 1 }( X_t^{ ( a_i ) }, X_{ t - 1 }^{ ( i ) } ),
\end{equation*}
that is, the entries of $\a_t | \a_{ t + 1 }, \X_{ t + 1 }, \X_t, \X_{ t - 1 }, \w_{ t - 1 }$ are independent, with 
\begin{align*}
&a_t^{ ( i ) } | \a_{ t + 1 }, \X_{ t + 1 }, \X_t, \X_{ t - 1 }, \w_{ t - 1 } \\
&\sim \operatorname{Categorical}( g_t( X_{ t + 1 }^{ ( a_{ t + 1 }^{ ( 1 ) } ) }, X_t^{ ( 1 ) } ) q_{ t - 1 }( X_t^{ ( 1 ) }, X_{ t - 1 }^{ ( i ) } ), \ldots, g_t( X_{ t + 1 }^{ ( a_{ t + 1 }^{ ( N ) } ) }, X_t^{ ( N ) } ) q_{ t - 1 }( X_t^{ ( N ) }, X_{ t - 1 }^{ ( i ) } ) ),
\end{align*}
where in this and subsequent uses, the probabilities parametrising categorical distributions are given up to a normalising constant.

We call a function $f$ $I$-increasing if it is increasing in $\sum_{ i \in I } | \{ j \in [ N ]: a_t^{ ( j ) } = i \} |$ for $I \subseteq [ N ]$.
A balls-in-bins coupling shows that an $I$-increasing $\E[ f( \a_t ) | \a_{ t + 1 }, \X_{ t + 1 }, \X_t, \X_{ t - 1 }, \w_{ t - 1 } ] \leq \E[ f( \tilde{ \a }_t ) ]$, where $\tilde{ \a }$ is independent of $\calF_{ \infty }$, the entries of $\tilde{ \a }_t$ are independent of each other, and
\begin{equation*}
\tilde{ a }_t^{ ( j ) } \sim \operatorname{Categorical}\Big( \Big( \frac{ a }{ \varepsilon } \Big)^{ \mathds{ 1 }_{ \{ 1 \in I \} } - \mathds{ 1 }_{ \{ 1 \notin I \} } }, \ldots, \Big( \frac{ a }{ \varepsilon } \Big)^{ \mathds{ 1 }_{ \{ N \in I \} } - \mathds{ 1 }_{ \{ N \notin I \} } } \Big),
\end{equation*}
which follows from substituting upper bounds from \eqref{weight_bound} and \eqref{mixing_bound} for the probabilities corresponding to bins in $I$, corresponding lower bounds elsewhere, and canceling common factors.
Writing
\begin{equation*}
\E[ c_N( t ) | \calF_{ t - 1 } ] = \frac{ 1 }{ ( N )_2 } \sum_{ i = 1 }^N \E[ ( \nu_t^{ ( i ) } )_2 | \calF_{ t - 1 } ] =: \frac{ 1 }{ ( N )_2 } \sum_{ i = 1 }^N \E[ f_i( \a_t ) | \calF_{ t - 1 } ],
\end{equation*}
noting that $f_i$ is $\{ i \}$-increasing, and using the binomial moment formula, 
\begin{equation}\label{mosimann}
X \sim \operatorname{Bin}( N; p ) \Rightarrow \E[ ( X )_q ] = ( N )_q p^q, 
\end{equation}
which can be found, for example, in \cite{Mosimann62}, applied to $\nu_t^{ ( i ) }$, yields
\begin{equation*}
\E[ c_N( t ) | \calF_{ t - 1 } ] \leq \frac{ 1 }{ ( N )_2 } \sum_{ i = 1 }^N ( N )_2 \Big( \frac{ a / \varepsilon }{ ( N - 1 ) \varepsilon / a + a / \varepsilon } \Big)^2 \leq \frac{ a^4 }{ N \varepsilon^4 }.
\end{equation*}
Flipping the upper and lower bounds in the argument establishes \eqref{c_N_bound} via
\begin{equation}\label{c_N_lb}
\E[ c_N( t ) | \calF_{ t - 1 } ] \geq \frac{ \varepsilon^4 }{ N a^4 }.
\end{equation}

To verify \eqref{c_N_2_bound}, we write
\begin{align}
\E[ c_N( t )^2 | \calF_{ t - 1 } ] &= \frac{ 1 }{ [ ( N )_2 ]^2 } \E\Bigg[ \Bigg( \sum_{ i = 1 }^N ( \nu_t^{ ( i ) } )_2 \Bigg)^2 \Big| \calF_{ t - 1 } \Bigg] \nonumber \\
&= \frac{ 1 }{ [ ( N )_2 ]^2 } \Bigg( \sum_{ i = 1 }^N \E[ [ ( \nu_t^{ ( i ) } )_2 ]^2 | \calF_{ t - 1 } ] + \sum_{ i = 1 }^N \sum_{ j \neq i }^N \E[ ( \nu_t^{ ( i ) } )_2 ( \nu_t^{ ( j ) } )_2 | \calF_{ t - 1 } ] \Bigg) \nonumber \\
&\leq  \frac{ 1 }{ [ ( N )_2 ]^2 } \sum_{ i = 1 }^N \Bigg\{ \E[ ( \nu_t^{ ( i ) } )_4 + 4 ( \nu_t^{ ( i ) } )_3 + 2 ( \nu_t^{ ( i ) } )_2 | \calF_{ t - 1 } ] \nonumber \\
&\phantom{\leq  \frac{ 1 }{ [ ( N )_2 ]^2 } }+ \sum_{ j \neq i }^N \E[ ( \nu_t^{ ( i ) } + \nu_t^{ ( j ) } )_4 + 4 ( \nu_t^{ ( i ) } + \nu_t^{ ( j ) } )_3 + 2 ( \nu_t^{ ( i ) } + \nu_t^{ ( j ) } )_2 | \calF_{ t - 1 } ] \Bigg\}, \label{c_N_2_ub}
\end{align}
where the inequality uses $[ ( \nu )_2 ]^2 = ( \nu )_4 + 4 ( \nu )_3 + 2 ( \nu )_2$.
The first expectation on the right-hand side is $\{ i \}$-increasing, and so by \eqref{weight_bound}, \eqref{mixing_bound} and \eqref{mosimann},
\begin{equation}\label{i_inc}
\E[ ( \nu_t^{ ( i ) } )_4 + 4 ( \nu_t^{ ( i ) } )_3 + 2 ( \nu_t^{ ( i ) } )_2 | \calF_{ t - 1 } ] \leq \frac{ ( N )_2 a^4 }{ N^2 \varepsilon^4 } \Big( \frac{ a^4 }{ \varepsilon^4 } + \frac{ 4 a^2 }{ \varepsilon^2 } + 2 \Big).
\end{equation}
The second is $\{ i, j \}$-increasing, which gives
\begin{equation}\label{ij_inc}
\E[ ( \nu_t^{ ( i ) } + \nu_t^{ ( j ) } )_4 + 4 ( \nu_t^{ ( i ) } + \nu_t^{ ( j ) } )_3 + 2 ( \nu_t^{ ( i ) } + \nu_t^{ ( j ) } )_2 | \calF_{ t - 1 } ] \leq \frac{ 8 ( N )_2 a^4 }{ N^2 \varepsilon^2 } \Big( \frac{ 2 a^4 }{ \varepsilon^4 } + \frac{ 4 a^2 }{ \varepsilon^2 } + 1 \Big). 
\end{equation}
Substituting \eqref{i_inc} and \eqref{ij_inc} into \eqref{c_N_2_ub} yields
\begin{equation*}
\E[ c_N( t )^2 | \calF_{ t - 1 } ] \leq  \frac{ 16 a^4 }{ ( N )_2 \varepsilon^4 } \Big( \frac{ a^4 }{ \varepsilon^4 } + \frac{ 2 a^2 }{ \varepsilon^2 } + 1 \Big) \leq \frac{ C_2 }{ N } \E[ c_N( t ) | \calF_{ t - 1 } ],
\end{equation*}
where the last inequality follows from \eqref{c_N_lb}.

Finally, for \eqref{D_N_bound} we write
\begin{align*}
\E[ D_N( t ) | \calF_{ t - 1 } ] &= \frac{ 1 }{ N ( N )_2 } \sum_{ i = 1 }^N \Big\{ \E[ ( \nu_t^{ ( i ) } )_2 \nu_t^{ ( i ) } | \calF_{ t - 1 } ] + \frac{ 1 }{ N } \sum_{ j \neq i }^N \E[ ( \nu_t^{ ( i ) } )_2 ( \nu_t^{ ( j ) } )^2 | \calF_{ t - 1 } ] \Big\} \\
&\leq \frac{ 1 }{ N ( N )_2 } \sum_{ i = 1 }^N \Big\{ \E[ ( \nu_t^{ ( i ) } )_3 + 2 ( \nu_t^{ ( i ) } )_2 | \calF_{ t - 1 } ] + \frac{ 1 }{ N } \sum_{ j \neq i }^N \E[ ( \nu_t^{ ( i ) } + \nu_t^{ ( j ) } )_4 \\
&\phantom{\leq \frac{ 1 }{ N ( N )_2 } \sum_{ i = 1 }^N \Big\{} + 5 ( \nu_t^{ ( i ) } + \nu_t^{ ( j ) } )_3 + 4 ( \nu_t^{ ( i ) } + \nu_t^{ ( j ) } )_2 | \calF_{ t - 1 } ] \Big\},
\end{align*}
where the second line follows from $( \nu )_2 \nu^2 = ( \nu )_4 + 5 ( \nu )_3 + 4 ( \nu )_2$.
The expectations are $\{ i \}$- and $\{ i, j \}$-increasing, respectively, and so \eqref{mosimann} gives
\begin{align*}
\E[ ( \nu_t^{ ( i ) } )_3 + 2 ( \nu_t^{ ( i ) } )_2 | \calF_{ t - 1 } ] &\leq \frac{ ( N )_2 a^4 }{ N^2 \varepsilon^4 } \Big( \frac{ a^2 }{ \varepsilon^2 } + 2 \Big), \\
\E[ ( \nu_t^{ ( i ) } + \nu_t^{ ( j ) } )_4 + 5 ( \nu_t^{ ( i ) } + \nu_t^{ ( j ) } )_3 + 4 ( \nu_t^{ ( i ) } + \nu_t^{ ( j ) } )_2 | \calF_{ t - 1 } ] &\leq \frac{ 8 ( N )_2 a^4 }{ N^2 \varepsilon^4 } \Big( 2 + \frac{ 5 a^2 }{ \varepsilon^2 } + \frac{ 2 a^4 }{ \varepsilon^4 } \Big).
\end{align*}
Using \eqref{c_N_lb} then yields
\begin{align*}
\E[ D_N( t ) | \calF_{ t - 1 } ] &\leq \frac{ 1 }{ ( N )_2 } \Big\{ \frac{ ( N )_2 a^4 }{ N^2 \varepsilon^4 } \Big( \frac{ a^2 }{ \varepsilon^2 } + 2  \Big) + \frac{ 8 ( N )_2 a^4 }{ N^2 \varepsilon^4 } \Big( 2 + \frac{ 5 a^2 }{ \varepsilon^2 } + \frac{ 2 a^4 }{ \varepsilon^4 } \Big) \Big\} \\
&\leq \frac{ C_1 }{ N } \E[ c_N( t ) | \calF_{ t - 1 } ].
\end{align*}
\end{proof}

\begin{cor}\label{multinomial_thm}
Genealogies of $n$ particles from SMC algorithms with multinomial resampling  converge to the Kingman $n$-coalescent in the sense of finite-dimensional distributions under the time-scaling introduced in Theorem \ref{general_thm} under the conditions of Lemma \ref{conditional_lemma}.
\end{cor}
\begin{proof}
The standing assumption holds by exchangeability of multinomial resampling.
Condition \eqref{binary_bound} is immediate by taking expectations in \eqref{c_N_bound}.

To verify \eqref{binary_bound_2}, we use Lemma \ref{optional_stopping_lemma} twice with \eqref{c_N_2_bound} in between to obtain
\begin{align*}
\E\Bigg[ \sum_{ r = \tau_N( s ) + 1 }^{ \tau_N( t ) } c_N( r )^2 \Bigg] &= \E\Bigg[ \sum_{ r = \tau_N( s ) + 1 }^{ \tau_N( t ) } \E[ c_N( r )^2 | \calF_{ r - 1 } ] \Bigg] \leq \frac{ C_2 }{ N } \E\Bigg[ \sum_{ r = \tau_N( s ) + 1 }^{ \tau_N( t ) } c_N( r ) \Bigg] \leq \frac{ C_2 ( t - s + 1 ) }{ N }.
\end{align*}

Condition \eqref{big_merger_bound} can be checked using \eqref{D_N_bound} via the same argument as for \eqref{binary_bound_2}:
\begin{align*}
\E\Bigg[ \sum_{ r = \tau_N( s ) + 1 }^{ \tau_N( t ) } D_N( r ) \Bigg] \leq \frac{ C_1 }{ N } \E\Bigg[ \sum_{ r = \tau_N( s ) + 1 }^{ \tau_N( t ) } c_N( r ) \Bigg] \leq \frac{ C_1 ( t - s + 1 ) }{ N }.
\end{align*}

Finally, for \eqref{tau_bound} we use \eqref{c_N_bound} and Lemma \ref{optional_stopping_lemma} to obtain
\begin{align*}
\E[ \tau_N( t ) - \tau_N( s ) ] = \E\Bigg[ \sum_{ r = \tau_N( s ) + 1 }^{ \tau_N( t ) } \frac{ \E[ c_N( r ) | \calF_{ r - 1 } ] }{ \E[ c_N( r ) | \calF_{ r - 1 } ] } \Bigg] \leq \frac{ a^4 }{ \varepsilon^4 } N ( t - s + 1 ).
\end{align*}
\end{proof}

We now demonstrate that while the mode of convergence in Theorem \ref{general_thm} is too weak for convergence of expectations of continuous, bounded test functions, useful information can still be obtained.
For example, the time until the Kingman $n$-coalescent reaches its MRCA can be constructed as $T_n := \sum_{ k = 2 }^n S_k$, where the $( S_2, \ldots S_n )$ are independent, and $S_k \sim \operatorname{Exp}( \binom{ k }{ 2 } )$.
Moments of $\tau_N( T_n )$ give the time scale on which a SMC algorithm will reach its MRCA.
\vskip 11pt
\begin{cor}\label{time_scale_cor}
Under the assumptions of Corollary \ref{multinomial_thm}, the following  bounds hold for any $1 \leq n \leq N$, and any coupling of $( \tau_N, T_n )$:
\begin{align*}
\frac{ 2 \varepsilon^4 N }{ a^4 } ( 1 - n^{ -1 } ) &\leq \E[ \tau_N( T_n ) ] \leq \frac{ 2 a^4 N }{ \varepsilon^4 } ( 1 - n^{ -1 } ) + \frac{ a^8 }{ \varepsilon^4 }, \\
\Var( \tau_N( T_n ) ) &\leq \frac{ N^2 a^8 }{ \varepsilon^8 } \Big( \frac{ 4 \pi^2 }{ 3 } - 12 + O( n^{ -1 } ) \Big) + O( N ), \\
\Var( \tau_N( T_n ) ) &\geq \frac{ N^2 \varepsilon^8 }{ a^8 } \Big( \frac{ 4 \pi^2 }{ 3 } - 12 + O( n^{ -1 } )\Big),
\end{align*}
\end{cor}
\begin{proof}
By \eqref{c_N_bound}, Lemma \ref{optional_stopping_lemma}, the definition of $\tau_N$, and the fact that
\begin{equation*}
\E[ c_N( t ) ] = \E[ \E[ c_N( t ) | \w_t ] ] = \sum_{ i = 1 }^N \E[ ( w_t^{ ( i ) } )^2 ] \leq \frac{ a^4 }{ N },
\end{equation*} 
we have
\begin{align*}
\E[ \tau_N( t ) ] &= \E\Bigg[ \sum_{ s = 1 }^{ \tau_N( t ) } \frac{ \E[ c_N( s ) | \calF_{ s - 1 } ] }{ \E[ c_N( s ) | \calF_{ s - 1 } ] } \Bigg] \leq N \frac{ a^4 }{ \varepsilon^4 } \E\Bigg[ \sum_{ s = 1 }^{ \tau_N( t ) } \E[ c_N( s ) | \calF_{ s - 1 } ] \Bigg] \\
&= N \frac{ a^4 }{ \varepsilon^4 } \E\Bigg[ \sum_{ s = 1 }^{ \tau_N( t ) } c_N( s ) \Bigg] \leq N \frac{ a^4 }{ \varepsilon^4 } \{ t + \E[ c_N( \tau_N( t ) ) ] \} \leq N \frac{ a^4 }{ \varepsilon^4 } \Big( t + \frac{ a^4 }{ N } \Big).
\end{align*}

A corresponding lower bound of $N t \varepsilon^4 / a^4$ is obtained via the same argument.
Conditioning on $T_n$ in $\E[ \tau_N( T_n ) ]$, and using $\E[ T_n ] = 2 ( 1 - n^{ -1 } )$ \cite[page 76]{wakeley09} establishes the claimed bounds on the mean.

A similar argument for the variance gives
\begin{align*}
\Var( \tau_N( t )  ) &= \Var\Bigg( \sum_{ s = 1 }^{ \tau_N( t ) } \frac{ \E[ c_N( s ) | \calF_{ s - 1 } ] }{ \E[ c_N( s ) | \calF_{ s - 1 } ] } \Bigg) \leq \Var\Bigg( \frac{ N a^4 }{ \varepsilon^4 } \sum_{ s = 1 }^{ \tau_N( t ) } \E[ c_N( s ) | \calF_{ s - 1 } ] \Bigg) \\
&= \frac{ N^2 a^8 }{ \varepsilon^8 } \Bigg( \E\Bigg[ \sum_{ s_1, s_2 = 1 }^{ \tau_N( t ) } c_N( s_1 ) c_N( s_2 ) \Bigg] - \E\Bigg[ \sum_{ s = 1 }^{ \tau_N( t ) } c_N( s ) \Bigg]^2 \Bigg) \\
&\leq \frac{ N^2 a^8 }{ \varepsilon^8 } \Big( \E[ \{ t + c_N( \tau_N( t ) ) \}^2 ] - t^2 \Big) \leq \frac{ 2 N a^{ 12 } }{ \varepsilon^8 } t + O( 1 ),
\end{align*}
because taking expectations in \eqref{c_N_2_bound} shows that $\E[ c_N( t )^2 ] = O( N^{ -2 } )$.
Thus, by the law of total variance,
\begin{align*}
\Var( \tau_N( T_n ) ) &= \Var( \E[ \tau_N( T_n ) | T_n ] ) + \E[ \Var( \tau_N( T_n ) | T_n ) ] \\
&\leq \frac{ N^2 a^8 }{ \varepsilon^8 } \Var( T_n ) + \frac{ 2 N a^{ 12 } }{ \varepsilon^8 } \E[ T_n ] + O( 1 ) \\
&= \frac{ N^2 a^8 }{ \varepsilon^8 } \Big( \frac{ 4 \pi^2 }{ 3 } - 12 + O( n^{ -1 } ) \Big) + O( N ),
\end{align*}
since $\Var( T_n ) = 4 \pi^2 / 3 - 12 + O( n^{ -1 } )$ \cite[page 76]{wakeley09}.
The other direction is much simpler since variance is nonnegative:
\begin{align*}
\Var( \tau_N( T_n ) ) &\geq \Var( \E[ \tau_N( T_n ) | T_n ] ) \geq \frac{ N^2 \varepsilon^8 }{ a^8 } \Var( T_n ) \\
&= \frac{ N^2 \varepsilon^8 }{ a^8 } \Big( \frac{ 4 \pi^2 }{ 3 } - 12 + O( n^{ -1 } ) \Big).
\end{align*}
\end{proof}
Simulations in the next section confirm that the scalings predicted by Corollary \ref{time_scale_cor} hold for real algorithms.
Strengthening the mode of convergence in Theorem \ref{general_thm} to obtain a wider class of bounds is a subject of ongoing work.

\section{A numerical example}\label{numerics}

In this section we study the robustness of Theorem \ref{general_thm} by demonstrating via simulation that the scalings of Corollary \ref{time_scale_cor} hold nonasymptotically for a particle system for which \eqref{weight_bound} and \eqref{mixing_bound} fail.
We also show that the same scalings hold for popular alternatives to multinomial resampling which do not satisfy the standing assumption.

Let $( X_t, Y_t )_{ t \geq 0 }$ be the discretised Ornstein--Uhlenbeck process:
\begin{align*}
X_{ t + 1 } &= ( 1 - \Delta ) X_t + \sqrt{ \Delta } \xi_t, \\
 X_0 &\sim N( 0, 1 ), \\
 Y_t | X_t &\sim N( X_t, \sigma^2 ),
\end{align*}
where $\Delta > 0$ is the step size, $\sigma^2$ is the observation noise, and $\bm{ \xi } \stackrel{\text{iid}}{\sim} N( 0, 1 )$.

We observe a realisation of the trajectory $\bm{ Y } = \bm{ y }$, but not $\bm{ X }$, and specify a bootstrap particle filter targeting the smoothing distribution $P( \bm{ x } | \bm{ y } )$ via
\begin{align*}
q_t( x, x' ) \equiv p( x, x' ) &:= ( 2 \pi \Delta )^{ -1 / 2 } \exp\Big( \frac{ - ( x' - ( 1 - \Delta ) x )^2 }{ 2 \Delta } \Big), \\
g_t( x, x' ) \equiv \psi( x', y_t ) &:= ( 2 \pi \sigma^2 )^{ -1 / 2 } \exp\Big( \frac{ - ( y_t - x' )^2 }{ 2 \sigma^2 } \Big), \\
p( x_{ -1 }, x ) \equiv \mu( x ) &:= ( 2 \pi )^{ -1 / 2 } \exp\Big( \frac{ - x^2 }{ 2 } \Big),
\end{align*}
in Algorithm \ref{particle_filter}.
We set $\Delta = \sigma = 0.1$, simulated an observed trajectory of length $T = 40960$, and used it as the input for bootstrap particle filters with $N = 8192$, and recorded subtree heights for uniformly sampled subsets of leaves of size $n \in \{ 2, 4, 8, \ldots, N \}$.
The mean and variance of tree heights were then estimated from 1000 replicates for each of multinomial, residual, stratified and systematic resampling, with all four simulations run using the same observed data and the same random seed. 
Representative results are shown in Figure \ref{mean_T_biD_N}; results for filters as small as $N = 128$ were similar.

All depicted moment estimators remain bounded away from 0 and $\infty$ as predicted by Corollary \ref{time_scale_cor}, including in the $n \approx N$ regime to which Theorem \ref{general_thm} does not apply.
As such, our simulations suggests that the $O( N \log N )$ bound of \cite{Jacob15} on the height of the genealogy of all $N$ particles could be sharpened to an $O( N )$ bound.
\begin{figure}[t]
\centering
\includegraphics[width = 0.48 \linewidth]{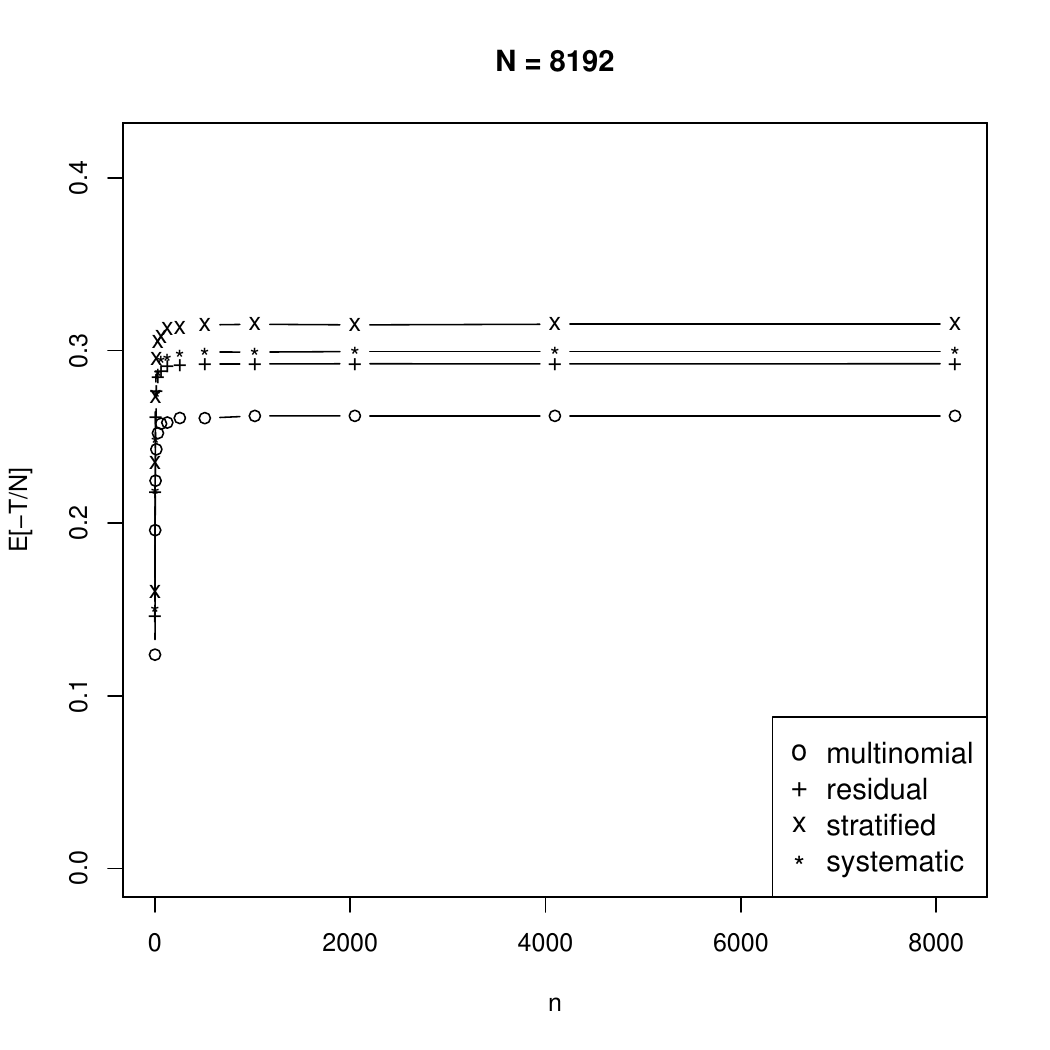}
\includegraphics[width = 0.48 \linewidth]{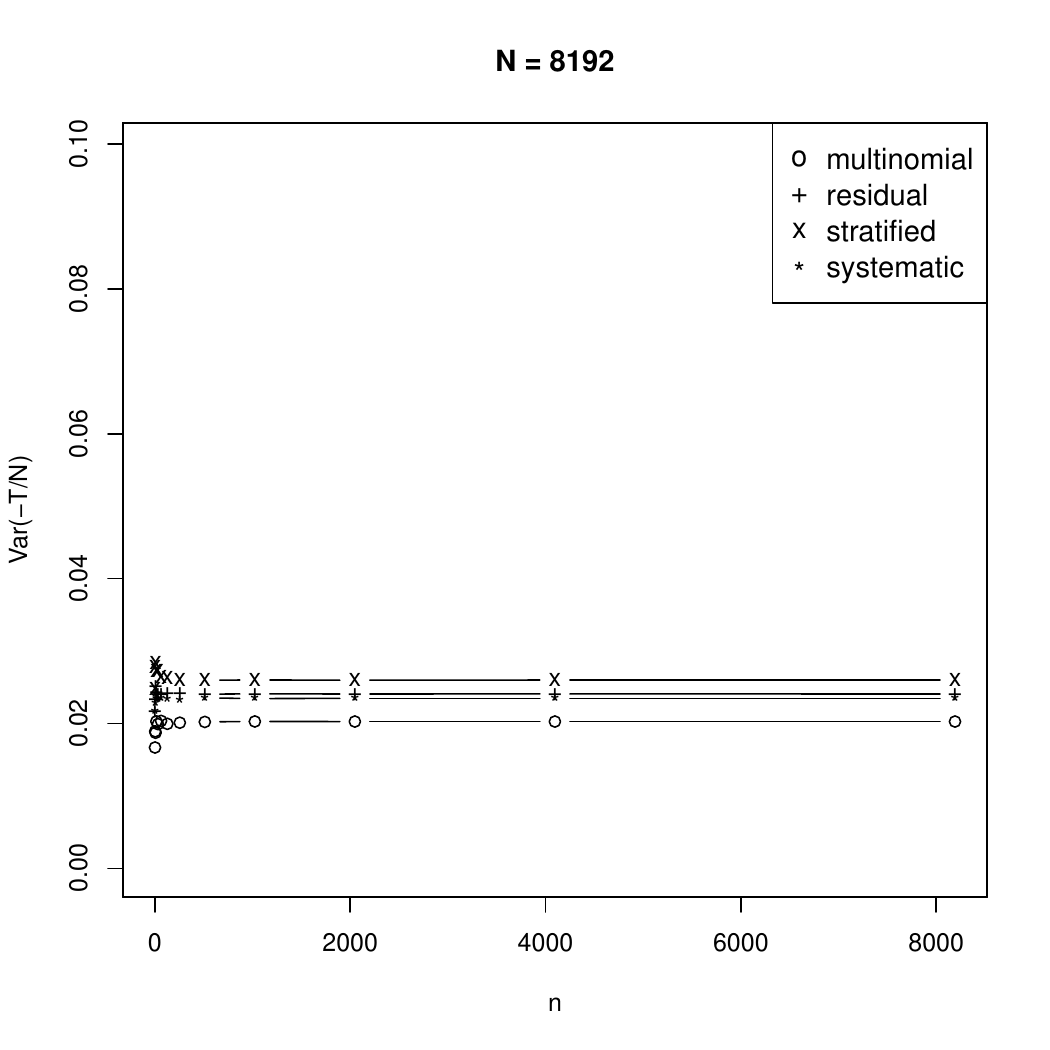}
\caption{Mean (left) and variance (right) of tree heights for $n \leq N$.}
\label{mean_T_biD_N}
\end{figure}

\section{Discussion}\label{discussion}

We have shown in Theorem \ref{general_thm} that genealogies of SMC algorithms converge, in the sense of finite-dimensional distributions, to the Kingman $n$-coalescent under a suitable rescaling of time, and certain assumptions.
Thus, we are able to use the tractability of the Kingman $n$-coalescent to characterise SMC algorithms as well.
We illustrated this in Corollary \ref{time_scale_cor} by obtaining scaling results for the first two moments of the number of generations from the leaves to the MRCA.
Asymptotic expressions for other quantities, for example, the probability of retaining at least two branches for a time window of a desired length can also be easily obtained.
Our method is the only tool for \emph{a priori} estimation of path degeneracy of which we are aware, and represents an important step towards practical guidelines for simultaneous minimisation of path and weight degeneracies.
Strengthening the mode of convergence in Theorem \ref{general_thm}, along the lines of \cite[Theorem 3.1]{moehle99}, would be desirable to enable the analysis of larger classes of functionals, and is the subject of ongoing work.

Our assumptions essentially require a compact state space and multinomial resampling.
However, the simulation study in Section \ref{numerics} suggests that neither is necessary: our example did not satisfy \eqref{weight_bound} or \eqref{mixing_bound}, and a range of resampling methods fit the predictions of Corollary \ref{time_scale_cor}.
We believe that our assumptions can be relaxed if subtler arguments are used to control the variability of family sizes.
Our simulations also confirmed that the predicted scalings hold nonasymptotically, and hence are relevant for real algorithms. 

Theorem \ref{general_thm} also demonstrates that the domain of attraction of the Kingman $n$-coalescent includes non-Markovian genealogies.
Previous results have focused on Markovian genealogies, which are typical models for neutral genetic evolution \cite[Theorem 1]{Moehle98}.
Thus, our results improve the tractability of models with random but inherited fecundity.

\section*{Appendix}

Here we verify that the Fubini and dominated convergence theorems apply to \eqref{binary_ub_2} and \eqref{pre-fubini_lb}, and prove Lemma \ref{optional_stopping_lemma}.
Taking the modulus of summands on the right-hand side of \eqref{binary_ub_2}, and using $D_N( t) \leq c_N( t )$ and \eqref{time_ub} yields
\begin{align*}
\lim_{ N \rightarrow \infty } \E\Bigg[ \prod_{ d = 1 }^k \chi_d \Bigg] \leq {}& \lim_{ N \rightarrow \infty } \sum_{ \beta_1 = 0 }^{ \infty } \ldots \sum_{ \beta_k = 0 }^{ \infty } ( 1 + O( N^{ -1 } ) )^{ | \bm{ \alpha } | + | \bm{ \beta } | } \\
&\times \prod_{ d = 1 }^k ( | Q |^{ \alpha_d + \beta_d } )_{ \eta_{ d - 1 } \eta_d } \frac{ \big\{ [ \binom{ n - 1 }{ 2 } + 1 ] ( t_d - t_{ d - 1 } + 1 ) \big\}^{ \alpha_d + \beta_d } }{ ( \alpha_d + \beta_d ) ! } \\
= {}& \prod_{ d = 1 }^k ( e^{ | Q | \big[ \binom{ n - 1 }{ 2 } + 1 \big] ( t_d - t_{ d - 1 } + 1 ) } )_{ \eta_{ d - 1 } \eta_d } < \infty,
\end{align*}
where $| Q |$ is the matrix whose entries are the absolute values of those of $Q$, and the last equality holds because $( | Q |^j )_{ \eta_{ d - 1 } \eta_d } = 0$ for $j < \alpha_d$.

Similarly, for \eqref{pre-fubini_lb}, using $D_N( t ) \leq c_N( t )$, \eqref{time_ub} and \eqref{path_counting} gives
\begin{align*}
\lim_{ N \rightarrow \infty } \E\Bigg[ \prod_{ d = 1 }^k \chi_d \Bigg] \leq {}& \lim_{ N \rightarrow \infty } \sum_{ \beta_1 = 0 }^{ \infty } \ldots \sum_{ \beta_k = 0 }^{ \infty } \binom{ n - 2 }{ 2 }^{ | \bm{ \alpha } | } \frac{ ( 1 + O( N^{ -1 } ) )^{ | \bm{ \alpha } | + | \bm{ \beta } | } }{ ( 1 + O( N^{ -2 } ) )^{ | \bm{ \alpha } | + | \bm{ \beta } | } } \nonumber \\
&\times \prod_{ d = 1 }^k ( | Q |^{ \alpha_d + \beta_d } )_{ \eta_{ d - 1 } \eta_d } \frac{ ( t_d - t_{ d - 1 } + 1 )^{ \alpha_d + \beta_d } }{ ( \alpha_d + \beta_d ) ! } \\
= {}& \binom{ n - 2 }{ 2 }^{ | \bm{ \alpha } | } \prod_{ d = 1 }^k ( e^{ | Q | ( t_d - t_{ d - 1 } + 1 ) } )_{ \eta_{ d - 1 } \eta_d } < \infty.
\end{align*}

\begin{proof}[Proof of Lemma \ref{optional_stopping_lemma}]
Define $M_{ s, t } := \sum_{ r = s + 1 }^t c_N( r ) - \E[ c_N( r ) | \calF_{ r - 1 } ]$, and for fixed $K > 0$ note that $\tau_N( t ) \wedge K$ is a bounded $\calF_t$-stopping time.
We have
\begin{align*}
\E[ M_{ s, \tau_N( t ) \wedge K } ] &= \sum_{ r = s + 1 }^K\E[ ( c_N( r ) - \E[ c_N( r ) | \calF_{ r - 1 } ] ) \mathds{ 1 }_{ \{ \tau_N( t ) \wedge K \geq r \} } ] \\
&= \sum_{ r = s + 1 }^K\E[ \mathds{ 1 }_{ \{ \tau_N( t ) \wedge K > r - 1 \} } \E[ ( c_N( r ) - \E[ c_N( r ) | \calF_{ r - 1 } ] ) | \calF_{ r - 1 } ] ] = 0,
\end{align*}
where the second line holds because $\mathds{ 1 }_{ \{ \tau_N( t ) \wedge K \geq r \} } = \mathds{ 1 }_{ \{ \tau_N( t ) \wedge K > r - 1 \} }$ is $\calF_{ r - 1 }$-measurable.
Conditioning on $\tau_N( s )$ and using $\tau_N( s ) \leq \tau_N( t )$ yields
\begin{equation*}
\E\Bigg[ \sum_{ r = \tau_N( s ) + 1 }^{ \tau_N( t ) \wedge K } c_N( r ) \Bigg] = \E\Bigg[ \sum_{ r = \tau_N( s ) + 1 }^{ \tau_N( t ) \wedge K } \E[ c_N( r ) | \calF_{ r - 1 } ] \Bigg],
\end{equation*}
and the monotone convergence theorem concludes the proof.
\end{proof}

\begin{lem}\label{indicator_lemma}
With notation as in Theorem \ref{general_thm}, and assuming the standing assumption as well as \eqref{big_merger_bound} and \eqref{binary_bound_2},
\begin{align*}
\lim_{ N \to \infty }\E\Bigg[ &\Bigg\{ \prod_{ d = 1 }^k \mathds{ 1 }_{ \{ \tau_N( t_d ) - \tau_N( t_{ d - 1 } ) \geq \alpha_d + \beta_d \} } \Bigg\}  \Bigg\{ \prod_{ r = \tau_N( t_0 ) + 1 }^{ \tau_N( t_d ) } \mathds{1}_{ \{ c_N( r ) < \binom{ n }{ 2 }^{ -1 }  \} } \mathds{ 1 }_{ \{ c_N( r ) > \binom{ n - 2 }{ 2 } D_N( r ) \} } \Bigg\} \Bigg] = 1.
\end{align*}
\end{lem}
\begin{proof}
By De Morgan's law, the union bound, and the Markov inequality,
\begin{align*}
&\E\Bigg[\Bigg\{ \prod_{ d = 1 }^k \mathds{ 1 }_{ \{ \tau_N( t_d ) - \tau_N( t_{ d - 1 } ) \geq \alpha_d + \beta_d \} } \Bigg\} \times \Bigg\{ \prod_{ r = \tau_N( t_0 ) + 1 }^{ \tau_N( t_d ) } \mathds{1}_{ \{ c_N( r ) < \binom{ n }{ 2 }^{ -1 }  \} } \mathds{ 1 }_{ \{ c_N( r ) > \binom{ n - 2 }{ 2 } D_N( r ) \} } \Bigg\} \Bigg] \\
&\geq 1 - \sum_{ d = 1 }^k \frac{ 1 }{ t_d - t_{ d - 1 } } \sum_{ r = 1 }^{ \alpha_d } \E[ c_N( \tau_N( t_{ d - 1 } + r ) ) ] - \binom{ n }{ 2 } \E\Bigg[ \sum_{ r = \tau_N( t_0 ) + 1 }^{ \tau_N( t_k ) } c_N( r )^2 \Bigg] \\
&\phantom{\geq 1} - \E\Bigg[ \sum_{ r = \tau_N( t_0 ) + 1 }^{ \tau_N( t_k ) } \mathds{ 1 }_{ \{ c_N( r ) \leq \binom{ n - 2 }{ 2 } D_N( r ) \} } \Bigg]
\end{align*}
where the first line of the right hand side tends to 1 as $N \to \infty$ similarly to \eqref{indicator_lb}: $\E[c_N(\tau_N(t))]$ vanishes by Lemma \ref{stopping_time_lemma}, and $\E[ \sum_{ r = \tau_N( t_0 ) + 1 }^{ \tau_N( t_k ) } c_N(r)^2 ]$ vanishes by \eqref{binary_bound_2}.
It remains to show the last expectation vanishes as $N \to \infty$.

Fix $0 < \varepsilon < 2 / ( 3 n^2 )$, and $N$ large enough that $1 / N < \varepsilon$.
Let $A_i( r ) := \{ \nu_r^{ ( i ) }  \leq N \varepsilon \}$.
Following \cite[Proof of Lemma 5.5]{Moehle03}, 
\begin{align*}
&\frac{ 1 }{ N ( N )_2 } \sum_{ i = 1 }^N ( \nu_r^{ ( i ) } )_2 \Bigg[ \nu_r^{ ( i ) } + \frac{ 1 }{ N } \sum_{ j \neq i } ( \nu_r^{ ( j ) } )^2 \Bigg] \mathds{ 1 }_{ A_i( r ) } \\
&\leq \frac{ 1 }{ N ( N )_2 } \sum_{ i = 1 }^N ( \nu_r^{ ( i ) } )_2 \Bigg[ N \varepsilon + \frac{ 1 }{ N } \sum_{ j = 1 }^N ( \nu_r^{ ( j ) } )_2 + \frac{ 1 }{ N } \sum_{ j = 1 }^N \nu_r^{ ( j ) } \Bigg] \mathds{ 1 }_{ A_i( r ) } \\
&\leq \Bigg[ \varepsilon c_N( r ) + \frac{ 1 }{ N } c_N( r ) + \frac{ N \varepsilon }{ N^2 ( N )_2 } \sum_{ i = 1 }^N \nu_r^{ ( i ) } \sum_{ j = 1 }^N ( \nu_r^{ ( j ) } )_2 \Bigg] \mathds{ 1 }_{ A_i( r ) } \\
&\leq \Bigg( 2 \varepsilon + \frac{ 1 }{ N } \Bigg) c_N( r ),
\end{align*}
and
\begin{equation*}
\frac{ 1 }{ N ( N )_2 } \sum_{ i = 1 }^N ( \nu_r^{ ( i ) } )_2 \Bigg[ \nu_r^{ ( i ) } + \frac{ 1 }{ N } \sum_{ j \neq i } ( \nu_r^{ ( j ) } )^2 \Bigg] \mathds{ 1 }_{ A_i( r )^c } \leq \sum_{ i = 1 }^N \mathds{ 1 }_{ A_i( r )^c }.
\end{equation*}
Thus
\begin{equation*}
\Bigg\{ c_N( r ) \leq \binom{ n - 2 }{ 2 } D_N( r ) \Bigg\} \subseteq \Bigg\{ \Bigg( \binom{ n - 2 }{ 2 }^{ -1 } - 2 \varepsilon - 1 / N \Bigg) c_N( r ) \leq \sum_{ i = 1 }^N \mathds{ 1 }_{ A_i( r )^c } \Bigg\},
\end{equation*}
and
\begin{equation*}
\bigcup_{ r = 1 }^N A_i( r )^c \Rightarrow c_N( r ) =  \frac{ 1 }{ ( N )_2 } \sum_{ i = 1 }^N ( \nu_t^{ ( i ) } )_2 > \frac{ \varepsilon ( N \varepsilon - 1 ) }{ N - 1 } > \varepsilon \Big( \varepsilon - \frac{ 1 }{ N } \Big) > 0.
\end{equation*}
Hence also
\begin{equation*}
\E\Bigg[ \sum_{ r = \tau_N( t_0 ) + 1 }^{ \tau_N( t_k ) } \mathds{ 1 }_{ \{ c_N( r ) \leq \binom{ n - 2 }{ 2 } D_N( r ) \} } \Bigg] \leq \E\Bigg[ \sum_{ r = \tau_N( t_0 ) + 1 }^{ \tau_N( t_k ) } \mathds{ 1 }_{ \big\{ \big( \binom{ n - 2 }{ 2 }^{ -1 } - 2 \varepsilon - 1 / N \big) \varepsilon ( \varepsilon - 1 / N ) \leq \sum_{ i = 1 }^N \mathds{ 1 }_{ A_i( r )^c }\big\} } \Bigg],
\end{equation*}
whereupon Lemma \ref{optional_stopping_lemma} and the conditional Markov inequality yield
\begin{align*}
&\E\Bigg[ \sum_{ r = \tau_N( t_0 ) + 1 }^{ \tau_N( t_k ) } \mathds{ 1 }_{ \{ c_N( r ) \leq \binom{ n - 2 }{ 2 } D_N( r ) \} } \Bigg] \\
&\leq \frac{ 1 }{ \big( \binom{ n - 2 }{ 2 }^{ -1 } - 2 \varepsilon - 1 / N \big) \varepsilon ( \varepsilon - 1 / N ) } \E\Bigg[ \sum_{ r = \tau_N( t_0 ) + 1 }^{ \tau_N( t_k ) } \sum_{ i = 1 }^N \mathds{ 1 }_{ A_i( r )^c } \Bigg].
\end{align*} 
A further two invocations of Lemma \ref{optional_stopping_lemma} with the conditional Markov inequality in between result in
\begin{align*}
&\E\Bigg[ \sum_{ r = \tau_N( t_0 ) + 1 }^{ \tau_N( t_k ) } \mathds{ 1 }_{ \{ c_N( r ) \leq \binom{ n - 2 }{ 2 } D_N( r ) \} } \Bigg] \\
&\leq \frac{ 1 }{ \big( \binom{ n - 2 }{ 2 }^{ -1 } - 2 \varepsilon - 1 / N \big) \varepsilon ( \varepsilon - 1 / N ) } \E\Bigg[ \sum_{ r = \tau_N( t_0 ) + 1 }^{ \tau_N( t_k ) } \sum_{ i = 1 }^N \frac{ ( \nu_r^{ ( i ) } )_3 }{ ( N \varepsilon )_3 } \Bigg] \\
&\leq \frac{ N ( N )_2 }{ \big( \binom{ n - 2 }{ 2 }^{ -1 } - 2 \varepsilon - 1 / N \big) \varepsilon ( \varepsilon - 1 / N ) ( N \varepsilon )_3 } \E\Bigg[ \sum_{ r = \tau_N( t_0 ) + 1 }^{ \tau_N( t_k ) } D_N( r ) \Bigg] \\
&\to \frac{ 1 }{ \big( \binom{ n - 2 }{ 2 }^{ -1 } - 2 \varepsilon\big) \varepsilon^5 } \times 0
\end{align*}
by \eqref{big_merger_bound} as $N \to \infty$, as required.
\end{proof}

\begin{lem}\label{stopping_time_lemma}
Assume \eqref{big_merger_bound} holds.
Then, for each $t > 0$,
\begin{equation*}
\lim_{ N \to \infty } \E[c_N( \tau_N( t ) ) ] = 0.
\end{equation*}
\end{lem}
\begin{proof}
Fix $\varepsilon > 0$ and let $N$ be large enough that $\varepsilon > 2 / N$.
Let $A_i( r ) := \{ \nu_r^{ ( i ) } \leq N \varepsilon \}$.
Following the same strategy as in the proof of Lemma \ref{indicator_lemma} above,
\begin{align*}
\frac{ 1 }{ ( N )_2 } \sum_{ i = 1 }^N ( \nu_r^{ ( i ) } )_2 ( \mathds{ 1 }_{ A_i( r ) } + \mathds{ 1 }_{ A_i( r )^c } ) &\leq \frac{ N \varepsilon }{ ( N )_2 } \sum_{ i = 1 }^N \nu_r^{ ( i ) } + \sum_{ i = 1 }^N \mathds{ 1 }_{ A_i( r )^c } \\
&\leq (1 + O( N^{ -1 } ) ) \varepsilon + \sum_{ i = 1 }^N \mathds{ 1 }_{ A_i( r )^c },
\end{align*}
and by Markov's inequality,
\begin{align*}
\E[ c_N( \tau_N( t ) ) ] &\leq ( 1 + O( N^{ -1 } ) ) \varepsilon + \sum_{ i = 1 }^N \frac{ \E[ ( \nu_{ \tau_N( t ) }^{ ( i ) } )_3 ] }{ ( N \varepsilon )_3 } \\
&\leq ( 1 + O( N^{ -1 } ) ) \varepsilon + \frac{ 1 + O( N^{ -1 } ) }{ \varepsilon^3 } \E\Bigg[ \sum_{ r = 1 }^{ \tau_N( t ) } D_N( r ) \Bigg] \to \varepsilon
\end{align*}
as $N \to \infty$, as required.
\end{proof}

\section*{Acknowledgements}

Jere Koskela was supported by EPSRC grant EP/HO23364/1 as part of the MASDOC DTC at the University of Warwick, by DFG grant BL 1105/3-2, and by EPSRC grant EP/R044732/1.
Paul Jenkins was supported in part by EPSRC grant EP/L018497/1.
Adam Johansen was partially supported by funding from the Lloyd's Register Foundation -- Alan Turing Institute Programme on Data-Centric Engineering.

\bibliographystyle{alpha}\bibliography{bibliography}  

\end{document}